\documentclass{article}[11pt]
\usepackage{geometry}                
\usepackage{amsmath}
\usepackage{leftidx}
\usepackage{paralist}
\usepackage{amsthm}
\usepackage{amsfonts}
\usepackage{amssymb}
\usepackage{epsfig}
\usepackage{graphicx}
\usepackage{epstopdf}
\usepackage{rotating}
\usepackage{color}

\newtheorem{theorem}{Theorem}

\newtheorem{lemma}{Lemma}
\newtheorem{proposition}{Proposition}

\newtheorem{definition}{Definition}
\newtheorem{corollary}{Corollary}

\renewcommand{\P}{{\mathcal{P}}}
\renewcommand{\O}{{\mathcal{O}}}
\newcommand{\C}{{\mathcal{C}}}
\renewcommand{\S}{{\mathcal{S}}}
\newcommand{\U}{{\mathcal{U}}}
\newcommand{\E}{{\mathcal{E}}}

\newcommand{\Z}{{\mathcal{Z}}}
\newcommand{\ts}{{\tilde{s}}}
\newcommand{\tc}{{\tilde{c}}}
\newcommand{\tS}{{\tilde{S}}}
\newcommand{\tC}{{\tilde{C}}}

\newcommand{\uM}{{\underline{M}}}
\newcommand{\oM}{{\overline{M}}}
\newcommand{\uN}{{\underline{N}}}
\newcommand{\oN}{{\overline{N}}}
\newcommand{\fz}{{\mathfrak{z}}}
\newcommand{\fy}{{\mathfrak{y}}}
\newcommand{\fp}{{\mathfrak{p}}}
\newcommand{\fP}{{\mathfrak{P}}}
\newcommand{\fw}{{\mathfrak{w}}}
\newcommand{\fs}{{\mathfrak{s}}}

\newcommand{\bN}{{\mathbb{N}}}
\newcommand{\bZ}{{\mathbb{Z}}}
\newcommand{\bI}{{\mathbb{I}}}

\renewcommand{\o}{{^o\! }}
\newcommand{\sZ}{{\sf{Z}}}

\begin{document}

\title{Reversibility and further properties of FCFS infinite bipartite matching  }

\author{Ivo {\sc Adan}\thanks{Department of Industrial Engineering \& Innovation Sciences,
Eindhoven University of Technology, P.O. Box 513,
5600 MB Eindhoven, the Netherlands, E-mail: {\tt i.j.b.f.adan@tue.nl}.
Research supported in part by Dutch stochastics cluster STAR.} 
\and 
Ana {\sc Bu\v si\'c}\thanks{INRIA/ENS, 23, avenue d'Italie, CS 81321, 75214 Paris Cedex 13, France. 
E-mail: {\tt ana.busic@inria.fr}.} 
\and Jean {\sc Mairesse}
\thanks{Sorbonne Universit\'es, UPMC Univ Paris 06, CNRS, LIP6, 4 place Jussieu, 75252 Paris Cedex 05, France. 
E-mail: {\tt jean.mairesse@lip6.fr}} 
\and Gideon {\sc Weiss}\thanks{
Department of Statistics,
The University of Haifa,
Mount Carmel 31905, Israel. E-mail: {\tt 
gweiss@stat.haifa.ac.il}. 
Research supported in part by
Israel Science Foundation Grants 711/09 and 286/13, and by Dutch stochastics cluster STAR.}}

\date{\today}

\maketitle

\begin{abstract}
The model of FCFS infinite bipartite matching was introduced in Caldentey, Kaplan, $\&$ Weiss {\em Adv. Appl. Probab.}, 2009.  In this model, there is a sequence of items that are chosen i.i.d. from a finite set $\C$ and an independent sequence of items that are chosen i.i.d. from  a finite set $\S$, and a bipartite compatibility graph $G$ between $\C$ and $\S$.  Items of the two sequences are matched according to the compatibility graph, and the matching is FCFS, meaning that each item in the one sequence is matched to the earliest compatible unmatched item in the other sequence.  In Adan $\&$ Weiss, {\em Operations Research}, 2012, a Markov chain associated with the matching was analyzed, a condition for stability was derived, and a product form stationary distribution was obtained. 
In the current paper, we present several new results that unveil the fundamental structure of the model.  First, we provide a pathwise Loynes' type construction which enables to prove the existence of a unique matching for the model defined over all the integers. Second, we prove that  the model is dynamically reversible: we define an exchange transformation in which we interchange the positions of each matched pair,
and show that the items in the resulting permuted  sequences are again independent and i.i.d., and the matching between them is FCFS in  reversed time.  
Third, we obtain product form stationary distributions of several new Markov chains associated with the model. As a by product, we compute useful performance measures, for instance the link lengths between matched items. 

\medskip

\medskip

\end{abstract}

\section{Introduction}
\label{sec.introduction}

%

Let us define the FCFS infinite bipartite matching model formally, following  \cite{caldentey-kaplan-weiss:09,adan-weiss:11}. 
\begin{definition}\label{de-ckw}
Consider two finite sets of {\em types} $\C=\{c_1,\ldots,c_I\}$ and $\S=\{s_1,\ldots,s_J\}$  and a bipartite {\em compatibility} graph $G=(\C,\S,\E)$, with $\E\subset \C\times \S$, assumed to be connected. 
Consider two independent random sequences $(c^m)_{m\in T}$ and $(s^n)_{n\in T}$ which are chosen respectively i.i.d. from $\C$ with probabilities $\alpha=\{\alpha_{c_1},\ldots,\alpha_{c_I}\}$, and  i.i.d. from $\S$ with probabilities $\beta=\{\beta_{s_1},\ldots,\beta_{s_J}\}$. The parameter set $T$ can be finite, for instance $T=\bI_N=0,1,2,\ldots,N$,   one sided infinite,  $T=\bN=0,1,2,3,\ldots$, or  two sided infinite, $T=\bZ=\ldots,-2,-1,0,1,2,3\ldots$.  We assume that $\alpha_i\neq 0$ for all $i$ and $\beta_j\neq 0$ for all $j$. The triple $(G,\alpha,\beta)$ defines an {\em infinite bipartite matching model}.
\end{definition}

In the compatibility graph, $(c_i,s_j)\in \E$ denotes that the types $c_i$ and $s_j$ are compatible and can be matched together.   If $(c_i,s_j) \not\in \E$, we say that $c_i$ and $s_j$ are incompatible and they cannot be matched to each other.
We  study {\em first come first served (FCFS) matchings} between the two sequences, as defined precisely in \cite{adan-weiss:11}, the exact definition is repeated in Section \ref{sec.FCFS}.    Loosely speaking it means that an item from $S$ is matched to the earliest possible compatible item from $C$, and vice versa.   
This is unambiguously defined when $T$ is finite or when $T=\bN$. In the case $T=\bZ$, defining precisely the meaning of the FCFS matching is the purpose of Section \ref{sec.loynes}. 

Figure \ref{fig.match1} illustrates  FCFS matching of two sequences, for a given compatibility graph.  Items are browsed from left to right. For instance, $c^1$ cannot be matched with $s^1$ nor $s^2$ since $c^1=c_1, s^1=s^2=s_2$ and $(c_1,s_2) \not\in \E$, but it will be matched with $s^3=s_3$ which is the earliest possible match. 

\begin{figure}[htb]
\begin{center}
\includegraphics[width=5.25in]{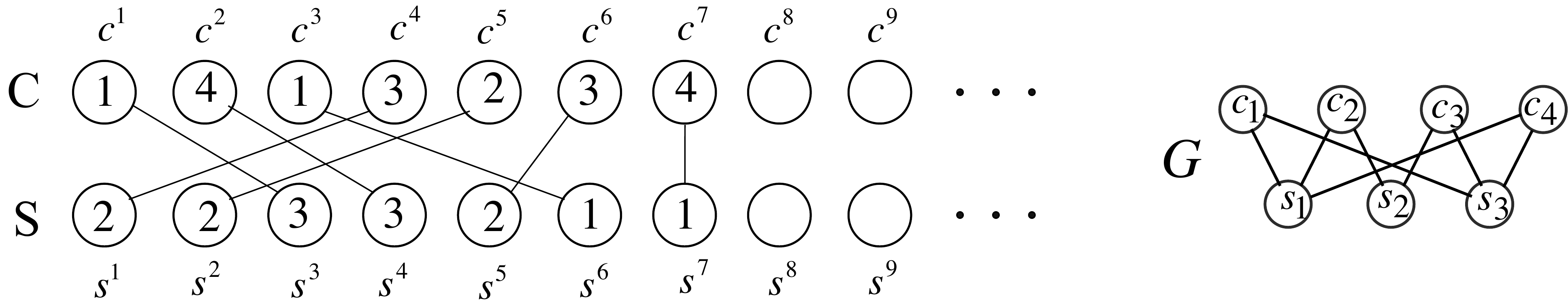}
\end{center}
\caption{FCFS matching of two sequences (left), for the compatibility graph $G$ (right)}
\label{fig.match1}
\end{figure}

{\bf Notations}:    Throughout the paper we shall use the following notations, terms, and illustrations to accompany our mathematical analysis. 
For convenience we shall refer to $s^n$ as servers, to $c^m$ as customers (we could equally well  call them, according to the context,  donors and patients, calls and agents, males and females, buy bids and sell bids), and we shall refer to $T$ as time (equivalently we could call it location, position, event).  We use $c_i,\,i=1,\ldots,I$ to denote the different types of customers, and we use $c^m$ to denote  the type of the $m$-th customer.   Similarly we use $s_j, \, j =1, \ldots, J$ to denote the types of servers, and $s^n$ denotes the type of the $n$-th server. In the accompanying figures, we will simplify the notation by writing  $c^m=i$ if $c^m=c_i$, and $s^n=j$ if $s^n=s_j$, and  we will arrange  the sequences in two lines, the top one containing the ordered customers, and the bottom one containing the ordered servers.  At a later stage, we will sometimes swap items between the top and the bottom lines in a way to be explained then.  In figures, we  put an edge between $s^n$ and $c^m$ if  they are matched, and we call this edge a {\em link} in the matching, characterized by the pair of times $(m,n)$.

\medskip

{\bf Motivations and previous works:}  In 1984, Ed Kaplan \cite{kaplan:84}  asked the following question when studying public housing in Boston:  Entitled families list those projects  where they wish to live, and when a home becomes available in a project, it is assigned by law to the longest waiting family that listed it.  The question then was what fraction of families with a specific list of preferences is actually assigned to each of the projects in their list. Kaplan called these fractions the {\em matching rates}.  The same mechanism may also  model  child adoptions, blood banks, organ transplants, and the operation of web servers and cloud computing. 

Kaplan's problem is related to the queueing model of parallel server scheduling, where there is a stream of arriving customers of several types $\C=\{c_1,\ldots,c_I\}$ which are served by several pools of parallel servers, with $n_1,\ldots,n_J$ servers of types $\S=\{s_1,\ldots,s_J\}$, and compatibility graph $G=(\C,\S,\E)$, where the policy  is FCFS, that is, whenever a server is available he attends to the longest waiting compatible customer.  This queueing model is notoriously hard to analyze, as demonstrated for instance in a paper by Foss and Chernova \cite{foss-chernova:98}, where it is proved that, with general service time distributions,  the stability of this model cannot be determined by the fluid model alone.  Motivated by Kaplan's problem, Whitt and Talreja \cite{talreja-whitt:07} studied a formal fluid model of the overloaded parallel server queueing model under FCFS policy, and derived explicit matching rates in the special cases of  bipartite graphs that are trees or  
can be decomposed as a tree of complete bipartite graphs.   

However, the parallel server queueing model may be inappropriate for the housing scenario, as apartments do not act as servers.  A more appropriate model is to assume that houses as well as families arrive in random streams, independent of each other, and, when a match is made, they depart the system.  
This led to the formulation, recalled in Definition \ref{de-ckw}, of the 
FCFS bipartite infinite matching model  by Caldentey, Kaplan and Weiss \cite{caldentey-kaplan-weiss:09}, which makes the further simplification of avoiding all consideration of timing of the independent arrival streams, and considers only the i.i.d. types of the ordered arrivals of customers and servers.   In \cite{caldentey-kaplan-weiss:09}, a necessary condition for stability was derived, and the FCFS matching rates were calculated for almost complete bipartite graphs.  The following questions were left open:  Is the condition for stability also sufficient, and can we calculate the matching rates for general bipartite graphs.  

Turning back to the parallel server queueing model, it turned out that the special case where arrivals are Poisson, service time distributions are  exponential, and the service rates  depend only on the type of server, and not on the type of customer that is served, the model under FCFS policy is tractable. Indeed, by specifying how customers are assigned to servers when there is a choice of several compatible available servers, it was shown in \cite{visschers-adan-weiss:12} that there exists a random assignment regime under which the queue lengths are  Markovian, satisfy partial balance equations, and have a product form stationary distribution.  This result was extended to a loss model in \cite{adan-hurkens-weiss:10,adan-weiss:12b}.  Admittedly, the very special form of the random assignment of customers to a choice of servers made this result of limited usefulness. But a similar product form result was later obtained in \cite{adan-weiss:14} for the parallel server queueing model with a more natural policy: the FCFS-ALIS policy where arriving customers that have a choice of several free compatible servers are assigned to the longest idle server (ALIS). 

An unexpected connection between the parallel server queueing model and the bipartite matching model was found in  \cite{adan-weiss:11} and provided a breakthrough for the bipartite matching model. 
Indeed,  it was proved in \cite{adan-weiss:11} that a Markov chain associated with the bipartite matching model, and directly inspired by the queue length process of the parallel server queueing model, enjoys the same remarkable properties as its source of inspiration: it satisfies partial balance, and has a product form stationary distribution.  As a by-product, it was obtained that the necessary conditions for stability were also sufficient, and, furthermore, if was possible to derive exact formula for the computation of the matching rates.  

Let us mention that further research on infinite bipartite  matching under policies different from FCFS, and on infinite matching in graphs that are not bipartite, are described in 
\cite{busic-gupta-mairesse:10} and \cite{mairesse-moyal:14}.   

To complete this review of previous work, let us mention another conjectural connection between the bipartite matching model and the parallel server queueing model. 
In \cite{adan-boon-weiss:13,adan-boon-weiss:14}, it is conjectured that in a parallel server queueing model, stabilized by abandonments, and operated under FCFS-ALIS, 
matching rates of the queueing system under many server scaling equal the matching rates of the FCFS infinite bipartite matching model.  This conjecture is as yet not proved, but simulations indicate that it may be correct, and that it can be used to evaluate performance and assist in design of such systems.

\medskip

{\bf New contributions of the paper}: Consider the FCFS infinite bipartite matching model. The following questions were left unanswered by the different contributions mentioned above. 
First, can one construct FCFS matching of two sequences that are defined for all the integers, $\mathbb{Z} = \ldots,-2,-1,0,1,\ldots$?   Second, partial balance usually indicates reversibility, can we show that the FCFS infinite matching model is time reversible in some sense?  Third, the Markov chain that was defined and analyzed in \cite{adan-weiss:11} is one of several Markov chains that can be associated with the FCFS infinite matching model, can we obtain stationary distributions for all of them (e.g. for the chain defined in \cite{caldentey-kaplan-weiss:09})?  We answer all these questions in the current paper.
Indeed, we uncover the fundamental structure of the model in the following three steps: 

\begin{compactitem}
\item
We derive a Loynes' scheme, which enables to get to stationarity through  sample path dynamics, and to prove the existence of a unique FCFS matching over $\mathbb{Z}$. 
\item
We define a pathwise transformation  in which we interchange the positions of the two items in a matched pair, see Figure \ref{fig.match2}, and we prove the ``dynamic reversibility''  of the model under this transformation. 

\item We construct  ``primitive'' Markov chains whose  product form stationary distributions are obtained directly from the dynamic reversibility. Using these as building blocks, we derive product form stationary distributions for multiple ``natural'' Markov chains associated with the model, and we compute various non-trivial performance measures as a by-product.
\end{compactitem}

\medskip

\begin{figure}[htb]
\begin{center}
\includegraphics[width=5.25in]{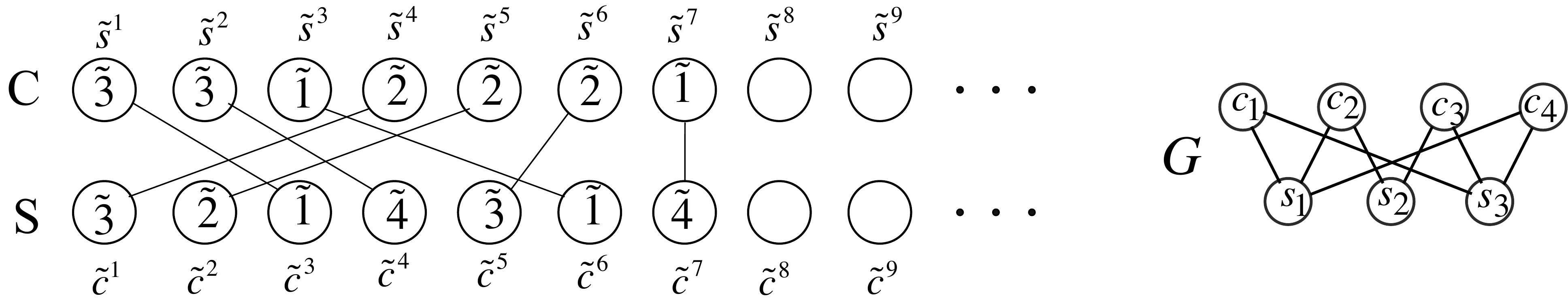}
\end{center}
\caption{The matched sequences of Figure \ref{fig.match1}, after the exchange transformation}
\label{fig.match2}
\end{figure}

{\bf Organization}: We now briefly describe the content of the sections of the paper.  
In Section \ref{sec.FCFS}, we state the definition of FCFS matching and summarize some results from \cite{adan-weiss:11}.  
These include uniqueness of the FCFS matching of two sequences over $\bN$, and two specific mechanisms to construct the  matching.  We define two Markov chains associated with these two mechanisms, which we call {\em the natural Markov chains}, and state the necessary and sufficient condition for their ergodicity, given by {\em complete resource pooling}.  

In Section \ref{sec.loynes}, we use a Loynes' type construction to obtain FCFS matching of two sequences over  $\bZ$.  We show that if complete resource pooling holds then the Loynes' construction,  matching all customers and servers in positions $n > -k$ and letting $k\to\infty$, converges almost surely to a unique FCFS matching.  This  also yields the unique stationary version of the Markov chains associated with the construction of the matching.

In Section \ref{sec.reversibility}, we discuss the {\em exchange transformation}, in which for a matched pair $c^m,s^n$, we exchange the customer and the server, so that we now have a customer in position $n$, denoted by $\tc^n$ (on the servers line) and a server in position $m$, denoted by $\ts^m$ (on the customers line),
as illustrated in Figure \ref{fig.match2}.  We show the following reversibility result:  starting from two independent i.i.d. sequences over $\bZ$ with FCFS matching between them, and performing the exchange transformation on all the links, we obtain two sequences of exchanged customers and servers and a matching between them.  It is then true that the sequences are again independent i.i.d., and the matching between them is FCFS in the reversed time direction.  

In Section \ref{sec.stationary}, we consider several  Markov chains that accompany the  construction of FCFS matching.  As suggested by the reversibility, these also have product from stationary distributions.  The stationary probabilities have an extremely simple structure, they behave like finite sequences of multi-Bernoulli trials.  We also derive an expression for the normalizing constant of the stationary distributions.  As is often the case, its computation is not easy, and we conjecture it is $\sharp P$.
We illustrate these results for the ``NN'' example that was described in \cite{caldentey-kaplan-weiss:09}, and that could not be fully analyzed at the time. 

Finally, in Section \ref{sec.perfomance}, we consider some performance measures for a stationary matching.  First, we recall the results from \cite{adan-weiss:11} on the computation of the matching rates 
$r_{c_i,s_j}$, that is, the frequency of matches of customer type $c_i$ with server type $s_j$.  Then we obtain novel results on another performance measure that is the random link length, $L_{c_i,s_j}$, given by $m-n$ for a match of $c^m$ with $s^n$, where $c^m=c_i,\,s^n=s_j$.  We obtain the stationary distribution of the link length, which is a mixture of convolutions of Geometric random variables.


\section{FCFS  bipartite matching over $\bN$}
\label{sec.FCFS}
In this section we provide a precise definition of bipartite FCFS matching, quote the result of uniqueness of the matching for $T=\bN$, describe two Markov chains to construct this unique matching, define complete resource pooling, and quote the result that these Markov chains are ergodic if and only if complete resource pooling holds.

\begin{definition}
\label{def.fcfs}
Let $c=(c^m)_{m\in T_1}$ and $s=(s^n)_{n\in T_2} $ be some fixed ordered sequences of customers and servers. 
Here $T_1$ and $T_2$  are index sets which are either finite, one-sided infinite or bi-infinite. 
\begin{compactitem}
\item[-]
A \emph{partial matching} of $c$ and $s$ is a set $A\subset T_1 \times T_2$ corresponding to 
customer-server pairs  satisfying:
\begin{eqnarray*}
&(i) & (m,n) \in A \implies (c^m,s^n)\in \E\\
& (ii) &\forall m , \ \# \{n \, : \, (m,n)\in A\} \leq 1, \ \forall n, \ \#\{m \, : \, (m,n) \in A\} \leq 1 \:. 
\end{eqnarray*}
For a partial matching $A$, we denote 
\[ A_c=\{m\,:\, \exists n, \, (m,n)\in A\},  \quad A_s=\{n\,:\, \exists m, \,(m,n)\in A\}\:.
\]
\item[-]
A partial matching $A$ is a \emph{(complete) matching}  if there are no unmatched compatible pairs left outside of $A$, that is:
\[
m\not\in A_c, \ n\not\in A_s \quad \implies \quad (c^m,s^n) \not\in \E.
\]
\item[-]
A matching $A$ is FCFS if for every $(m,n)\in A$,
\begin{eqnarray*}
&& \mbox{if $l < n,\, (c^m,s^l)\in \E$, then there exists $k<m$ such that $(k,l)\in A$ and } \\
&& \mbox{if $k < m,\, (c^k,s^n)\in \E$, then there exists $l<n$ such that $(k,l)\in A$.}
\end{eqnarray*}
Note that the condition simply checks for every match, that all earlier compatible items have been matched already, which is the FCFS property.
\item[-]
A  matching is \emph{perfect} if  $T_1=T_2$, and all customers and servers are matched.
\end{compactitem}
\end{definition}

The following Lemma was proved by induction in \cite{adan-weiss:11}.
\begin{lemma}[Adan \& Weiss, 2011 \cite{adan-weiss:11}]
\label{thm.unique}
For every finite index sets $\bI_M,\bI_N$, there exists a complete FCFS matching, and it is unique.
\end{lemma}
\begin{definition}
\label{def.fair}
The sequences $(c^m)_{m\in\bN}$ and $(s^n)_{n\in \bN}$ are {\em matchable} if the number of type $c_i$ customers and the number of type $s_j$ servers is infinite for all $i=1\ldots,I$ and $j=1,\ldots,J$.   The sequences 
$(c^m)_{m\in\bZ}$ and $(s^n)_{n\in \bZ}$ are {\em matchable} if the same also holds for times $-1,-2,\ldots$.   
\end{definition}

The following theorem was proved by Adan and Weiss \cite{adan-weiss:11}, we provide a sketch of the proof, for completeness.
\begin{theorem}
\label{thm.oneside}
For any two matchable sequences $(c^m)_{m\in\bN}$ and $(s^n)_{n\in \bN}$, there exists a unique perfect FCFS matching.  The matching can be obtained in a constructive way up to arbitrary length.
\end{theorem}
\begin{proof}
The uniqueness follows from Lemma \ref{thm.unique}.
If we matched all servers up to server $s^n$, then we can find an unmatched customer that will match $s^{n+1}$, because there is an infinite number of customers of each type, and similarly for $c^m,\,c^{m+1}$. 
By the uniqueness stated in Lemma \ref{thm.unique}, any construction in which complete $\mathbb{I}_M,\mathbb{I}_N$ matchings are extended by letting $M,N$ increase to infinity in any order will lead to the same infinite perfect matching over $\bN$ (we present three constructive methods in the next paragraphs).
\end{proof}

We  consider three methods of constructing a FCFS matching step by step, and define three Markov chains associated with them:
\begin{description}
\item[(i)] Matching pair by pair:
Proceeding from a complete FCFS matching of $(c^m,s^n)_{m,n\le N}$, we add the  pair $c^{N+1},s^{N+1}$,  and match them FCFS to compatible  previously unmatched servers and customers if possible, or to each other if possible, or else leave one or both unmatched.  With each step we associate a state that consists of the ordered lists of  the unmatched servers and customers.  It is easy to see that the step by step evolution of the state  defines a countable state discrete time irreducible and aperiodic Markov chain.  We denote it by $O=(O_N)_{N\in \bN}$, this is the `natural' pair by pair FCFS Markov chain.
\item[(ii)] Matching server by server:
Proceeding from the FCFS matching of all the servers $s^n,\, n\le N$, we add the next server $s^{N+1}$, and match it to the first compatible customer that has not yet been matched.   With each step we associate a state that consists of the ordered list of skipped customers.  This again defines a countable state discrete time irreducible and aperiodic Markov chain.  We denote it by $Q^s=(Q^s_N)_{N\in \bN}$, this is the `natural' server by server FCFS Markov chain.
\item[(iii)] Matching customer by customer is  analogous, resulting in a `natural' customer by customer FCFS Markov chain $Q^c=(Q^c_N)_{N\in \bN}$.
\end{description}

All three Markov chains have a common state if at time $N$ a perfect matching of all previous customers and servers is reached, in which case the state consists of empty lists and is denoted by $\emptyset$ or by $0$.  Because these Markov chains have a common state, they will be transient, null recurrent, or ergodic together.  We will define several additional Markov chains associated with the construction of FCFS matching in Section 
\ref{sec.stationary}.  All of them will be transient, null recurrent, or ergodic together, and have a common empty state.  
\begin{definition}
\label{def.ergodic}
We will say that the FCFS bipartite matching for a given $(G,\alpha,\beta)$ is {\em ergodic} if the corresponding Markov chains are ergodic. 
\end{definition}
We introduce some further notation to formulate the condition for ergodicity:   
We let $\S(c_i)$ be the set of server types compatible with customer type $c_i$, and $\C(s_j)$ be the set of customer types compatible with server type $s_j$.   For  $C\subset \C$ and $S\subset \S$, we define
\[
\S(C) = \bigcup_{c_i\in C} \S(c_i), \qquad \C(S) = \bigcup_{s_j\in S} \C(s_j) \:.
\]
We also define 
\[
\U(S) = \overline{\C(\overline{S})} = \C \setminus  \C(\S \setminus \ S)
\]
(where $\;\overline{\;\cdot\;}\;$ denotes the complement) which is the set of customer types that can only be served by the servers in $S$.  We call these the {\em unique customer types} of $S$.
Finally, we let  
\[
\alpha_{C} = \sum_{c_i\in C} \alpha_{c_i}, \qquad \beta_{S} = \sum_{s_j\in S} \beta_{s_j}\:.
\]
The following lemma follows immediately from the definitions.
\begin{lemma}\label{le-equiv}
The following three conditions are equivalent: 
\begin{eqnarray}
\label{eqn.pooling}
&\forall C\subset \C, C\ne\emptyset, C\ne \C, & \alpha_C < \beta_{\S(C)}  \nonumber \\
&\forall S\subset \S, S\ne\emptyset, S\ne \S, & \beta_S < \alpha_{C(S)} \\
&\forall S\subset \S, S\ne\emptyset, S\ne \S, &\beta_S > \alpha_{\U(S)}. \nonumber
\end{eqnarray}
\end{lemma}
In words, these three conditions ensure that there is enough service capacity for every subset of customer types, there is enough demand for service to ensure that every subset of servers can be fully utilized, and every subset of server types has enough capacity to serve customer types which are unique to this subset.
\begin{definition}
\label{def.crp}
We say that the infinite bipartite matching model $(G,\alpha,\beta)$ satisfies {\em complete resource pooling},  if the equivalent conditions in Lemma \ref{le-equiv} hold. 
\end{definition}
\begin{theorem}[Adan \& Weiss 2011 \cite{adan-weiss:11}]\label{th-ergodic}
The FCFS bipartite matching for $(G,\alpha,\beta)$ is ergodic if and only if complete resource pooling holds.
\end{theorem}
We will derive a simple confirmation of this result in Section \ref{sec.marginals}.

%


\section{Loynes' construction of FCFS bipartite matching over $\bZ$}
\label{sec.loynes}

In this section we show that if complete resource pooling holds, then for  two independent bi-infinite sequences of i.i.d. customers and servers, $(s^n,c^n)_{n\in \bZ}$ there exists almost surely a unique FCFS matching.  This matching coincides with the matchings obtained from the stationary versions of the various Markov chains described above.   The matching is obtained using a Loynes' type scheme (see \cite{loynes:62} for the original Loynes' construction). 

\begin{figure}[htb]
\begin{center}
\includegraphics[width=2.5in]{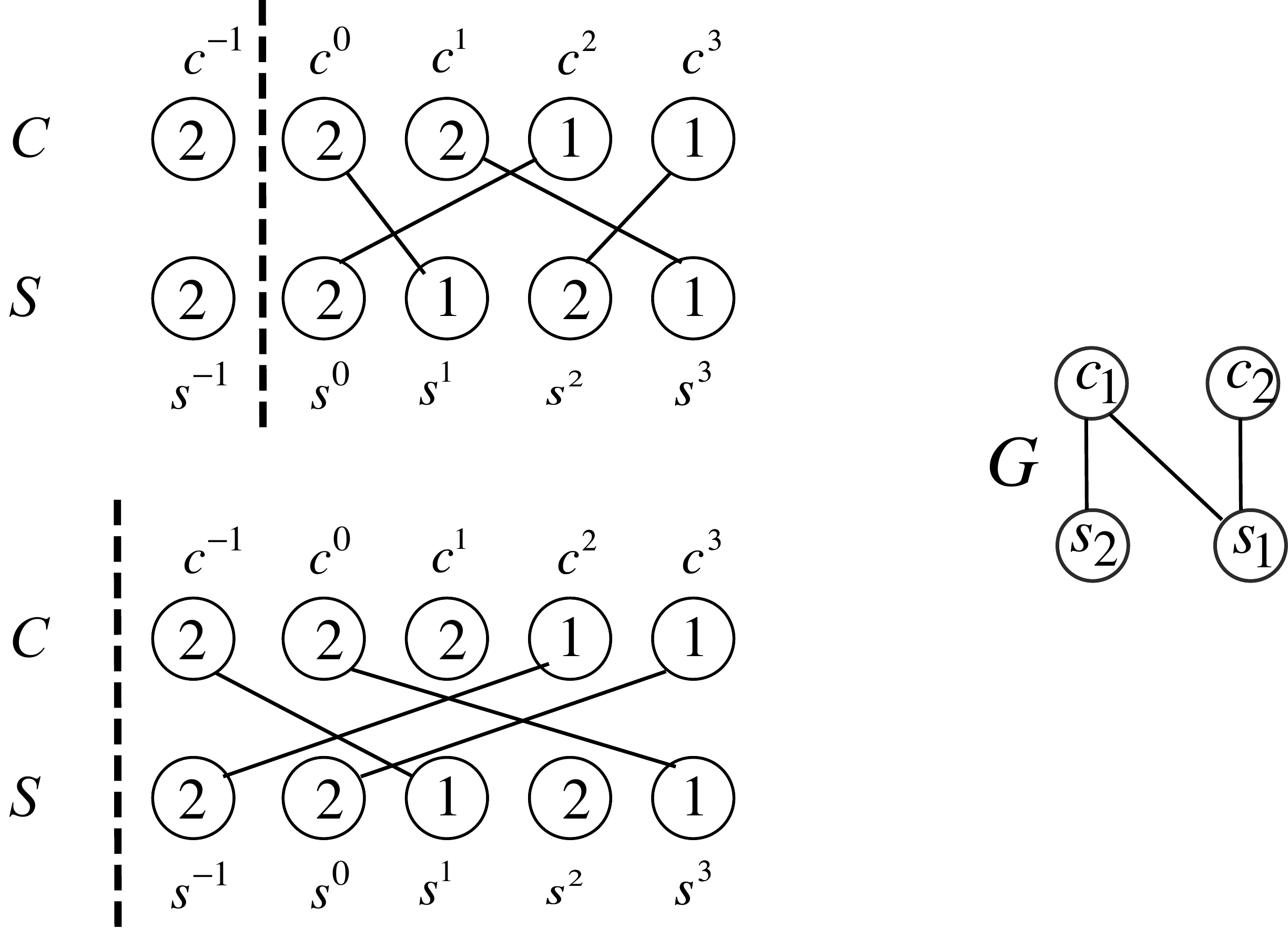}
\end{center}
\caption{Backward step changes the matching}
\label{fig.Loynes}
\end{figure}

Consider the (unique by Theorem \ref{thm.oneside}) matching of $(s^n,c^n)_{n > -K}$, and let $K\to \infty$.
At first glance we notice that as we let $K$ increase the matching changes. This is illustrated in the example of Figure \ref{fig.Loynes}. In this example, if $s^{-K-1}=s_2$ and $c^{-K-1}=c_2$, then starting empty at $-K-1$, the state at $-K$ cannot be empty, so the matching starting empty from $-K$ is different from the matching starting from $-K-1$.
Nevertheless, in this simple example, classical results can be used to prove that the matching does converge as $K\to\infty$.  The condition for that is that complete resource pooling should hold, which in this example happens when $\alpha_2 < \beta_1$.

We sketch the argument to provide some intuition for the general proof to come. Denote by $(O_n^{[-K]})_{n\geq -K}$, the pair by pair FCFS Markov chain associated with the sequence $(c^n,s^n)_{n\geq -K}$ (see the definition above). Denote by $X_n^{[-K]}$ the number of unmatched customers, all of them of type $c_2$, (or of  unmatched servers, all of them of type $s_2$) in $O_n^{[-K]}$. Observe that:
\[
X_{n+1}^{[-K]} = \begin{cases}
X_{n}^{[-K]} + 1& \mathrm{if} \quad (c^{n+1},s^{n+1}) = (c_2,s_2) \\
\max(X_{n}^{[-K]} - 1,0) & \mathrm{if} \quad (c^{n+1},s^{n+1}) = (c_1,s_1) \\
X_{n}^{[-K]} & \mathrm{if} \quad (c^{n+1},s^{n+1}) = (c_1,s_2)\ \mathrm{or} \ (c_2,s_1) \:.
\end{cases}
\]
So $(X_n^{[-K]})_{n\ge -K}$ can be interpreted as the queue-length process of a discrete-time version of an M/M/1 queue.  Complete resource pooling then implies that the drift of this queue is negative. A straightforward adaptation of the original Loynes argument (designed for a continuous time G/G/1 queue with negative drift), yields the existence of a limiting process for $(X_n^{[-K]})_{n\geq -K}$ when $K\rightarrow +\infty$, see for instance Chapter 2.1 in \cite{BaBr}. Let us denote the limiting process by $(X_n^{\infty})_{n\in \bZ}$. Again, the original Loynes argument shows that the set of indices $\{n\in\bZ, \ X_n^{\infty} = 0\}$ is a.s. infinite. These indices can be viewed as regeneration points for the processes $(X_n^{[-K]})_{n\ge -K}$. In particular, if $X_k^{\infty}=0$ then $X_k^{[-K]}=0$ for all $-K\leq k$. Now, observe that these regeneration points are also regeneration points for the processes $(O_n^{[-K]})_{n\geq -K}$, that is, $[X_k^{\infty}=0]\implies [\forall K, \ -K\leq k, \ O_k^{[-K]}=\emptyset]$.
Between two regeneration points, we have a finite complete matching of the sequences of customers and of servers, and this matching will not change over time. Therefore, there exists a limiting bi-infinite matching which is obtained by simply concatenating the finite matchings between regeneration points.

\medskip

The goal is to prove the same thing in general. We will prove the following result:


\begin{theorem}
\label{thm.co-loynes}
For two independent i.i.d. sequences $(s^n,c^n)_{n\in \bZ}$, if complete resource pooling holds, there exists almost surely a unique FCFS matching over all of $\bZ$, and it can be obtained by
 Loynes' scheme, of constructing a FCFS matching from $-K$ to $\infty$, and letting $K \to \infty$.
\end{theorem}

We prove the Theorem in several steps, requiring three lemmas and two propositions.
The next two lemmas are pathwise results which do not depend on any probabilistic assumption. 

\begin{lemma}[Monotonicity]
\label{thm.monotonicity}
Consider $c^1,\ldots,c^M$ and $s^1,\ldots,s^N$, and complete FCFS matching between them.  
Assume there are $K$  customers and $L$ servers left unmatched.  Consider now an additional customer $c^0$, and the complete FCFS matching between $c^0,c^1,\ldots,c^M$ and $s^1,\ldots,s^N$.  Then this matching will have no more than $K+1$ customers and $L$  servers unmatched. 
\end{lemma}
\begin{proof}
Denote $A=(c^1,\ldots,c^M)$ and $B=(s^1,\ldots,s^N)$.
In the matching of $(c^0,A)$ and $B$, if $c^0$ has no match, then all the other links in the matching are the same as in the matching of $A$ and $B$, so the total number of unmatched customers is $K+1$ and unmatched servers is $L$.  If $c^0$ is matched to $s^n$ and $s^n$ is unmatched in the matching of $A$ and $B$, then $(c^0,s^n)$ is a new link and all the other links in the matching of  $(c^0,A)$ and $B$ are the same as in the matching of $A,B$, so  the total number of unmatched customers is $K$ and unmatched servers is $L-1$.

If $c^0$ is matched to $s^{n_1}$ and $s^{n_1}$ was matched to $c^{m_1}$ in the $A,B$ matching,  then $(c^0,s^{n_1})$ is a new link, and the link $(s^{n_1},c^{m_1})$ in the $A,B$ matching is disrupted.  We now look for a match for  $c^{m_1}$ in the matching of $(c^0,A)$ and $B$.  Clearly, $c^{m_1}$ is not matched to any of $s^j,\,j<n_1$, since in the construction of the $A,B$ matching $c^{m_1}$ was not matched to any of those.  So $c^{m_1}$ will either remain unmatched, or it will be matched to some $s^{n_2}$, where $n_2>n_1$.  In the former case, all the links of the $A,B$ matching except $(s^{n_1},c^{m_1})$ remain unchanged in the $(c^0,A)$ and $B$ matching, and so the numbers of unmatched items remain $K+1$ and $L$.  In the latter case, there are again two possibilities:  If $s^{n_2}$ is unmatched in the $A,B$ matching, then the $(c^0,A),B$ matching will have disrupted one link and added 2 links retaining all other links of the $A,B$ matching, so the numbers of unmatched items are $K$ and $L-1$.  If $s^{n_2}$ is matched to $c^{m_2}$ in the $A,B$ matching, then the link $s^{n_2},c^{m_2}$ is disrupted, and we now look for a match for  $c^{m_2}$ in the $(c^0,A),B$ matching.  Similar to $c^{m_1}$, either $c^{m_2}$ remains unmatched, resulting in $K+1$ and $L$ unmatched items in the  
$(c^0,A),B$ matching, or, by the same argument as before, $c^{m_2}$ will be matched to $s^{n_3}$, where $n_3>n_2$.  Repeating these arguments for any additional disrupted links, we conclude that  we either end up with one more link, so the number of unmatched items are $K$ and $L-1$, or we have the same number of links and the number of unmatched items are $K+1$ and $L$.
\end{proof}

\begin{lemma}[Subadditivity]
\label{thm.subadditivity}
Let $A'=(c^1,\ldots,c^m),\,A''=(c^{m+1},\ldots,c^M)$ and $B'=(s^1,\ldots,s^n)$, $B''=(s^{n+1},\ldots,s^N)$ and let $A=(c^1,\ldots,c^M)$, $B=(s^1,\ldots,s^N)$.  Consider the complete FCFS matching of $A',B'$, of $A'',B''$, and of $A,B$ and  let $K',K'',K$ be the number of unmatched customers and $L',L'',L$ be the number of unmatched servers in these three matchings.  Then $K \le K'+K''$ and $L\le L'+L''$.
\end{lemma}
\begin{proof}
Let $\hat{A}'=(\hat{c}^1,\ldots,\hat{c}^{K'})$ and $\hat{B}'=(\hat{s}^1\ldots,\hat{s}^{L'})$ be the ordered unmatched customers and servers from the complete FCFS matching of $A',B'$.  Then the FCFS matching of $(\hat{A}',A'')$ and $(\hat{B}',B'')$ will have exactly the same ordered unmatched customers and servers as the FCFS matching of $A$ and $B$.   We now construct the matching of $(\hat{A}',A'')$ and $(\hat{B}',B'')$ in steps.  We look first at the matching of $(\hat{c}^{K'},A'')$ with $B''$, next the matching of $(\hat{c}^{K'},A'')$ with $(\hat{s}^{L'},B'')$, and so on, pre-pending each of $(\hat{c}^1,\ldots,\hat{c}^{K'})$ and $(\hat{s}^1\ldots,\hat{s}^{L'})$ in reverse order, to the current matching.  At each step, by Lemma \ref{thm.monotonicity}, the number of links is unchanged or increased by 1, and the number of unmatched items  goes down by one or is increased by 1.  It follows that the total number unmatched customers is $\le K'+K''$ and of unmatched servers is $\le L'+L''$.
\end{proof}

\begin{lemma}
\label{thm.decrease}
Consider an incompatible pair $c^0,s^0$.  Then there exists an $h$ and a sequence 
$c^1,\ldots,c^h$, $s^1,\ldots,s^h$ with $h\leq \min\{I,J\}-1$, where $(s^i,c^i),\,i=1,\ldots,h$ are compatible, such that the FCFS matching of  $c^0,c^1,\ldots,c^h,s^0,s^1,\ldots,s^h$ is perfect.  
\end{lemma}

Clearly, the probability of occurrence of such a sequence is strictly positive and lower bounded by: $\delta = \prod_{(c,s)\in \E} \alpha_{c} \beta_{s}$.

\begin{proof}
Because the bipartite graph is connected, and there is no direct edge between $c^0,s^0$, there exists a simple  path (i.e. with no repeated nodes), $c^0\to s_{j_1} \to c_{i_1} \cdots \to  s_{j_h} \to c_{i_h} \to s^0$ which connects them, with $1\le h \le \min\{I,J\}-1$.  Clearly, the FCFS matching of  $c^0,c^1,\ldots,c^h,s^0,s^1,\ldots,s^h$, where $c^l=c_{i_l},s^l=s_{j_l},\,l=1,\ldots,h$ is perfect, with exactly the links of the path, where $c^0$ is matched to $s^1$, and $s^0$ is matched to $c^h$.  Note that FCFS matching of $c^1,\ldots,c^h,s^1,\ldots,s^h$ consists of $h$ perfectly matched blocks of length one.   
%
\end{proof}

We assume from now on in this section that complete resource pooling holds. 
By Theorem \ref{th-ergodic}, the pair by pair matching Markov chain $(O_N)_{N\in \bN}$ is ergodic. Using the Kolmogorov extension theorem \cite{oksendal:03}, we may define (in a non-constructive way) a stationary version $O^*=(O^*_N)_{N\in \bZ}$ of the Markov chain. Define also $O^{[k]}=(O^{[k]}_N)_{N \ge -k}$ the realization of the Markov chain that starts at $O^{[k]}_{-k}=\emptyset$.

Our first task is to show forward coupling, namely that $O^*$ and $O^{[0]}$ coincide after a finite time $\tau$ with $E(\tau)<\infty$.  Following that we use standard arguments to show backward coupling and convergence to a unique matching.

\begin{proposition}[Forward coupling]
\label{thm.forward}
The two processes $(O^*_n)_{n\in \bN}$ and $(O^{[0]}_n)_{n\in \bN}$ will couple after a finite time $\tau$, with $E(\tau)<\infty$.
%
%
\end{proposition}

\begin{proof}  
Denote by $|z|$ the number of unmatched customers (equal to the number of unmatched servers), for any state $z$ of the pair by pair matching process, we refer to it as the length of the state.   
Consider the sequence of times $0 \le M_0 < M_1 < \cdots < M_\ell,\cdots$ at which $O^*_{M_\ell} = \emptyset$.  This sequence is infinite with probability 1, and $E(M_\ell) = E(M_0)+ \ell E(M_1-M_0) < \infty, \ell\ge 0$ by the ergodicity.  
Consider the state $z_0=O^{[0]}_{M_0}$.  Then $|z_0| \le M_0$.   
By the monotonicity result of Lemma \ref{thm.monotonicity} and Lemma \ref{thm.subadditivity}, the states of $O^{[0]}$ satisfy
$|z_0| \ge |O^{[0]}_{M_1}| \ge \ldots \ge |O^{[0]}_{M_\ell}|$, i.e. the  length of the state of $O^{[0]}$ at the times $M_0,M_1,\ldots$ is non-increasing.  This is because each block of customers and servers in times between $M_{\ell-1}$ and $M_\ell$ on its own has 0 unmatched.  Furthermore, there is a positive probability greater or equal to $\delta >0$ that every $\bar{h}=\min\{I,J\}-1$ steps along this sequence the length will actually decrease by at least 1, as shown in Lemma \ref{thm.decrease}.  This will happen if the blocks that follow $M_\ell$ are blocks of length 1 forming a path in the bipartite graph that connects an unmatched pair in $O^{[0]}_{M_\ell}$.  Such a sequence  of blocks will consist  of  $\le \bar{h}$ blocks, and it will occur with probability $\ge \delta$.
Hence, there will be coupling after at most  $\sum_{j=1}^{|z_0|} L_j \bar{h}$ perfect matching blocks of $O^*$, where $L_j$ are i.i.d. geometric random variables with $\delta$ probability of success.  So coupling occurs almost surely, and the coupling time $\tau$ satisfies $E(\tau) \le  E(M_0) \left( 1 + \frac{1}{\delta} \bar{h} E(M_1-M_0) \right)$.
\end{proof}

Note that once $O^{[0]}$ and $O^*$ couple, they stay together forever.
We now need to show backward coupling.   

\begin{proposition}[Backward coupling]
\label{thm.backcoupling}
Let $O^*$ be the stationary pair by pair FCFS matching process, and let $O^{[-k]}$ be the process starting empty at time $-k$.    Then $\lim_{k\to\infty} O^{[-k]}_N = O^*_N$ for all $N\in\bZ$ almost surely.
\end{proposition}
\begin{proof}
The statement of almost surely refers to the product measure $(\alpha\otimes\beta)^{\otimes \bZ}$ on the product space $(C\times S)^{\bZ}$. 
%

Define $T_k = \inf \{N\ge -k: O^{[-k]}_N = O^*_N \}$.  By the forward coupling Proposition \ref{thm.forward}, we get that $T_k$ is almost surely finite. 
Let $\hat{T}_K = \max_{0\le k \le K} T_k$.  It is  $\ge 0$, and is also almost surely finite 
for any $K$.  $\hat{T}_K$ is the time at which all the processes starting empty at time $-k$, where $0\le k\le K$, couple with $O^*$, and remain merged forever.
Define the event $E_K = \{ \omega : \forall n \ge 0,\, O^{[-n]}_{\hat{T}_K} = O^*_{\hat{T}_K} \}$, in words, those $\omega$ for which the process starting empty at any time before $0$, will merge with $O^*$ by time $\hat{T}_K$.  We claim that $P(E_K)>0$.  We evaluate $P(\overline{E}_K)$.  
For any fixed $n \geq 0$, call $E_{n,K}$ the event that $O^{[-n]}$ couples with $O^*$ by time $\hat{T}_K$. 
We have 
$E_K = \bigcap_{n \geq 0} E_{n,K} = \bigcap_{n > K} E_{n,K}$ (by definition of $\hat{T}_K$, $E_{n,K}$ is always true for $n\leq K$, so we only need to consider $n>K$), so $\overline{E_K} = \bigcup_{n > K} \overline{E_{n,K}}$.   

The event $\overline{E_{n,K}}$ will happen if starting at the last time prior to $-K$ at which the process $O^{[-n]}$ was empty,  the next time that it is empty is after time $0$.  
The reason for that is that otherwise the process $O^{[-n]}$ reaches state $\emptyset$ at some time  $k\in [-K,0]$ and from that time onwards it is coupled with  $O^{[-k]}$, and will couple with $O^*$ by time $\hat{T}_K$.   

Define for $n>K$, $D_n =\{\omega : O^{[-n]}_m \ne \emptyset \mbox{ for } -n < m \le 0\}$.  Clearly by the above, $\bigcup_{n > K} \overline{E_{n,K}} \subseteq \bigcup_{n > K} D_n$.  
Let $\tau$ denote the recurrence time of the empty state.  Then: 
\[
P(\overline{E_K}) = P(\bigcup_{n > K} \overline{E_{n,K}}) \le P(\bigcup_{n > K} D_n) \le 
\sum_{n > K} P(\tau > n).
\]
By the ergodicity $\sum_{l=0}^\infty P(\tau> l) =E(\tau)<\infty$.  Hence we have that $P(\overline{E_K}) \to 0$ as $K\to\infty$, and therefore  $P(E_K)>0$ for large enough $K$, and 
$P(E_K) \to 1$ as $K\to \infty$.  Note also that $E_K \subseteq E_{K+1}$.

Define now $\hat{T}= \sup_{k\ge 0}  T_k$.  We claim that $\hat{T}$ is finite a.s.  Consider any $\omega$. Then by $P(E_K) \to 1$ as $K\to \infty$ and by the monotonicity of $E_K$, almost surely for this $\omega$ there exists a value $\ell$ such that $\omega \in E_\ell$.  But if $\omega \in E_l$, then $\hat{T}(\omega) \le \hat{T}_\ell < \infty$.   

So,  all processes starting empty before time $0$ will couple with $O^*$ by time $\hat{T}$.  
By the stationarity of the sequences $(s^n,c^n)_{n\in\bZ}$ and of $O^*$,   
we then also have that all processes $O^{[-k]}_N$ starting empty before $-k$ will couple with $O^*$ by time 0, if $k\ge \hat{T}$.  
Hence using the Loynes' scheme of starting empty at $-k$ and letting $k\to\infty$ the constructed process will merge with $O^*$ at time 0.  But the same argument holds not just for 0, but for any negative time $-N$.   Hence $O^{[-k]}$ and $O^*$ couple at $-N$ (and stay coupled) for any $k > N+  \hat{T}$.  This completes the proof.
\end{proof}

\begin{proof}[Proof of Theorem \ref{thm.co-loynes}]
We saw that $\lim_{k\to\infty} O^{[-k]}_N=O^*_N$ for all $N$ almost surely.  Each process $O^{[-k]}_N$ determines matches uniquely for all $N> - k$, so if we fix $N$, matches from $N$ onwards are uniquely determined by $\lim_{k\to\infty} O^{[-k]}_N$.  Hence $(O^*_N)_{N\in\bZ}$ determines for every server $s^n$ and every customer $c^n$ his match, uniquely, almost surely.  This proves the  theorem. 
\end{proof}

\section{Exchange Transformation and Dynamic Reversibility}
\label{sec.reversibility}
\begin{figure}[htb]
\begin{center}
\includegraphics[width=2.5in]{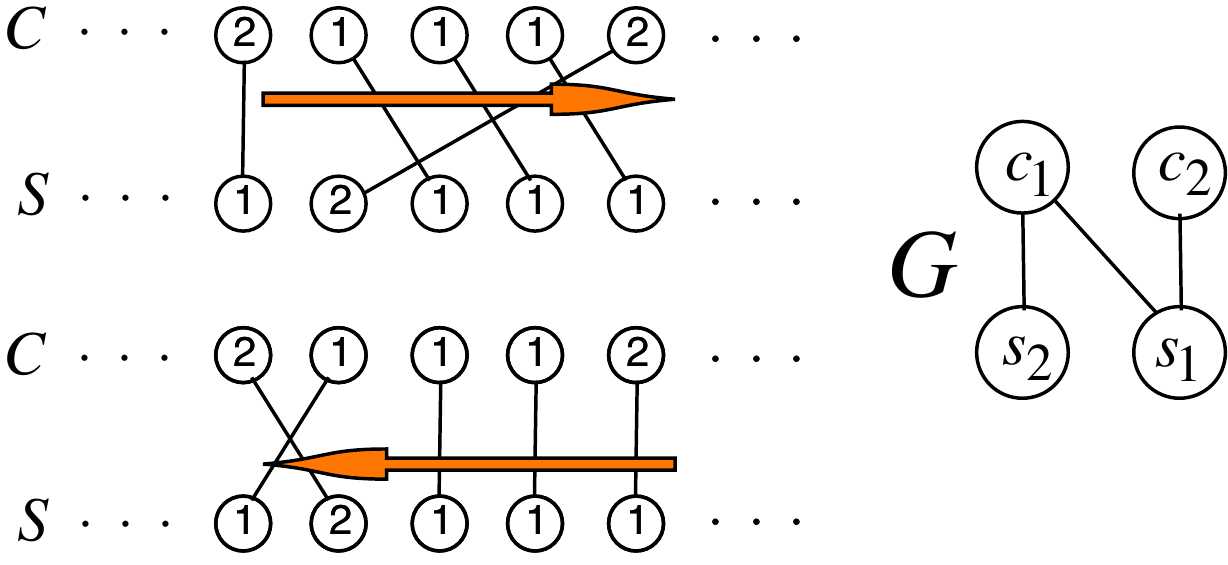}
\end{center}
\caption{FCFS is not preserved when time is reversed}
\label{fig.noreverse}
\end{figure}
The FCFS matching depends on the time direction in which it is constructed.  
The  simple example in Figure \ref{fig.noreverse}
 shows that FCFS is not preserved if the time direction is reversed. 

Nevertheless, this model has an elegant reversibility result.
In this section we introduce the exchange transformation, in which we switch the positions of each matched pair of customer and server.  Figures \ref{fig.match1} and \ref{fig.match2} illustrate the exchange transformation.  We show that the exchanged sequences are independent i.i.d. and that the matching is FCFS in reversed time.

\begin{definition}
Consider a FCFS bipartite matching of sequences $(s^n,c^n)_{n\in T}$.  The exchange transformation of the matched pair $s^n,c^m$, is the matched pair $\ts^m,\tc^n$ where $\ts^m=s^n$ and $\tc^n=c^m$.
\end{definition}

\begin{lemma}
\label{thm.reverse}
Let $A$ be a perfect matching of $c^1,\ldots,c^M$, and $s^1,\ldots,s^M$.  Let $\tc^1,\ldots,\tc^M$, $\ts^1,\ldots,\ts^M$ be the sequences obtained by the exchange transformation, retaining the same links of the matched pairs.  The resulting matching of  $\tc^1,\ldots,\tc^M$, $\ts^1,\ldots,\ts^M$  is the unique FCFS matching in reversed time.
\end{lemma}

\begin{proof}
Denote by $\tilde{A}$ the matching of $\tc^1,\ldots,\tc^M$, $\ts^1,\ldots,\ts^M$ obtained by retaining the links of $A$.  Then $(n,m)\in \tilde{A}$ if and only if $(m,n)\in A$.   Take $(n,m)\in \tilde{A}$, i.e. $\tc^n=c'$ and $\ts^m=s'$ are matched (and of course, $(s',c')\in \E$).   To show that $\tilde{A}$ is FCFS in reversed time, we need to show that:
\begin{eqnarray*}
&& \mbox{if $k >  m,\, (c',\ts^k)\in \E$, then there exists $l > n$ such that $(k,l)\in \tilde{A}$ and } \\
&& \mbox{if $l > n,  \, (\tc^l,s')\in \E$, then there exists $k > m$ such that $(k,l)\in \tilde{A}$  }.
\end{eqnarray*}

Assume to the contrary that there exists $k>m$ such that $\ts^k=s''$, and $(c',s'') \in \E$,  and there is no $l>n$ such that $(k,l)\in \tilde{A}$.  Because this is a complete matching, there exists an $l \le n$ such that $\tc^l$ is matched to   $\ts^k$, and because $\tc^n$ is matched to $\ts^m$, we actually have $l< n$.
So we have that $\ts^k=s''$ is matched to $\tc^l=c''$, where $k>m$ and $l<n$.  
Consider now the original matching.  We have that $c^m=c'$ is matched to $s^n=s'$, and we have $s^l=\ts^k=s''$ and $c^k=\tc^l=c''$ which are matched.  But, $(c',s'')\in \E$, and $l<n$ while $k>m$.  This contradicts the FCFS property of the matching $A$ of  the unexchanged sequences.  A similar contradiction is obtained if we assume the second part of the definition is violated.
The proof is illustrated in the following figure:
\begin{center}
\includegraphics[width=3.5in]{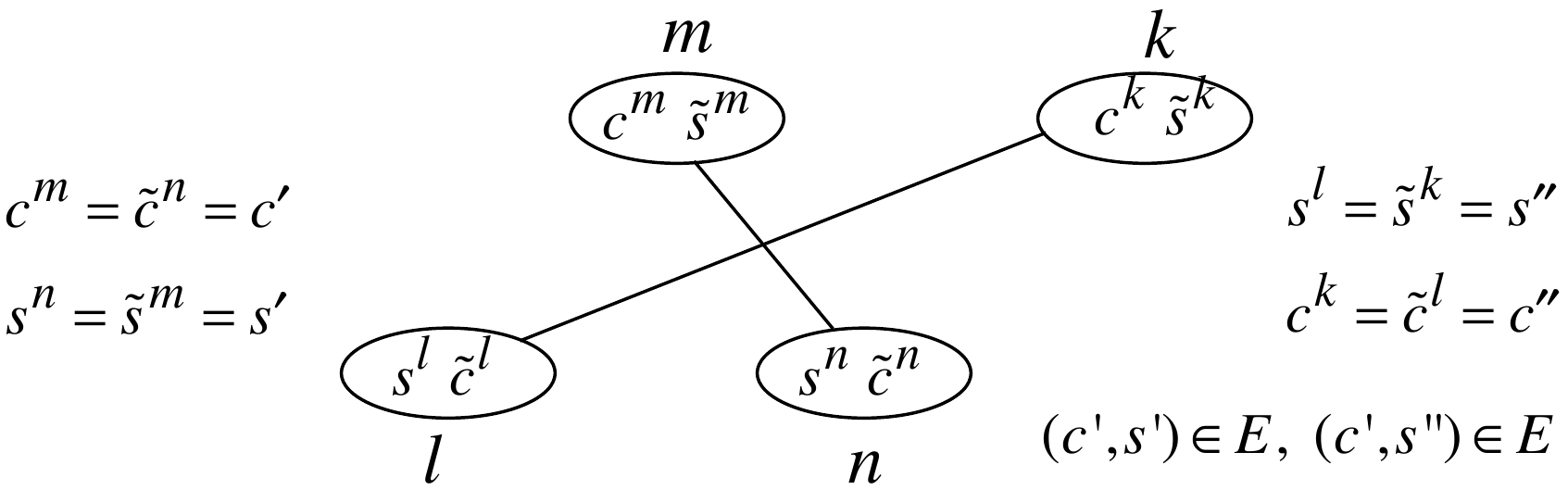}
\end{center}
\end{proof}

\begin{lemma}
\label{thm.blocks}
Consider the FCFS matching of two i.i.d. sequences, and let  $\O_{(m+1,m+M)}$ be the block of  customers and servers for the times   $[m+1,m+M]$.  Then the conditional probability of observing values $\O_{(m+1,m+M)} = \big((c^{m+1},\ldots,c^{m+M})$, $(s^{m+1},\ldots,s^{m+M})\big)$ conditional on the event that the FCFS matching of these values is a perfect match is: 
\begin{eqnarray*}
&& P\Big(\O_{m+1,M+M}= \big((c^{m+1},\ldots,c^{m+M}),\, (s^{m+1},\ldots,s^{m+M})\big) \\ 
&& \qquad \Big| \,
\O_{m+1,M+M} \mbox{ forms a perfect FCFS match} \Big) \\
&& \quad = \kappa_M    \prod_{i=1}^{I}  {\alpha_{c_i}}^{\# c_i} \prod_{j=1}^{J} {\beta_{s_j}}^{\# s_j}
\end{eqnarray*}
where $\kappa_M$ is a constant that may depend on $M$, and $\# c_i$, $\# s_j$ count the number of type $c_i$ customers and type $s_j$ servers in the block.
\end{lemma}
\begin{proof}
The conditional probability is calculated using Bayes formula:
\begin{eqnarray*}
&& \quad  P(\mbox{seeing $c^{m+1},\ldots,c^{m+M}, s^{m+1},\ldots,s^{m+M}$ }| \mbox{ having a perfect match}) \\
&& \quad = \frac{P(\mbox{having a perfect match }| \mbox{ seeing $c^{m+1},\ldots,c^{m+M}, s^{m+1},\ldots,s^{m+M}$})}
{P( \mbox{having a perfect match of length $M$})} \\
&& \qquad \times P(\mbox{seeing $c^{m+1},\ldots,c^{m+M}, s^{m+1},\ldots,s^{m+M}$}) \\
&& \quad = \kappa_M \times \mathbf{1}_{\{c^{m+1},\ldots,c^{m+M}, s^{m+1},\ldots,s^{m+M} \mbox{ is a perfect match\}}} 
 \times  \prod_{i=1}^{I}
{\alpha_{c_i}}^{\# c_i} \prod_{j=1}^{J} {\beta_{s_j}}^{\# s_j}
\end{eqnarray*}
where $\kappa_M=1\big/ P( \mbox{having a perfect match of length $M$})$.
\end{proof}

\begin{corollary}
\label{thm.symmetric} 
Let $\O_{m+1,M+M}$ be a FCFS perfectly  matched block, and 
let $\overleftarrow{\O}_{m+1,M+M}$ be obtained from $\O_{m+1,M+M}$ by performing the exchange transformation and time reversal.  Then $\overleftarrow{\O}^M$ is a FCFS perfectly matched block, and 
\[
P(\overleftarrow{\O}_{m+1,M+M}) = P(\O_{m+1,M+M})
\]
\end{corollary}
\begin{proof}
That $\overleftarrow{\O}_{m+1,M+M}$ is a FCFS perfectly matched  block follows from Lemma \ref{thm.reverse} and that $P(\overleftarrow{\O}_{m+1,M+M}) = P(\O_{m+1,M+M})$ follows from Lemma \ref{thm.blocks}.
\end{proof}

We now assume that the system is ergodic, i.e. the Markov chains $O_N,Q^s_N,Q^c_N$ are ergodic (which holds if and only if complete resource pooling holds).
We have shown in Theorem \ref{thm.co-loynes} that for independent i.i.d  sequences of servers and customers over $\bZ$, under ergodicity,  there exists a.s. a unique FCFS matching, which corresponds to the stationary version of the Markov chains associated with  the FCFS matching (generated by the Loynes' construction). 
Let $(c^n,s^n)_{n\in \Z}$ be the sequences of customers and servers, let $A$ be the FCFS matching. Define $A_n = A \cap \{(\cdot,n), (n,\cdot)\}$, that is, $A_n$ gives the matchings of the customer $c^n$ and of the server $s^n$. Consider the quadruple $(O_n,c^n,s^n,A_n)_{n\in \Z}$, 
denote by $\fp=(O_n,c^n,s^n,A_n)_{n\in \Z}$  the sample path corresponding to $(c^n,s^n)_{n\in\Z}$, and by $\mathfrak{P}$  its probability distribution.   Build the exchange transformation on  $\fp$  and consider it in reverse time. We obtain a new quadruple based on the two doubly infinite sequences  of customers and  servers $\tc^n,\ts^m,\,m,n \in \bZ$,  and a matching between them.  Denote the transformed path by $\psi\fp$.  Then $\psi\fp$ retains the links of $\fp$. 
Denote the probability distribution of the new quadruple by $\psi\fP$. Our goal is to prove that 
$\fP=\psi\fP$. 

Let $\mathfrak{P}^0$ be the Palm version of the measure $\mathfrak{P}$ with respect to $(O_n)_{n\in \Z}$, that is, $\mathfrak{P}^0$ is the law of $(O_n,c^n,s^n,A_n)_{n\in \Z}$ conditioned on the event $\{O_0=0\}$. 
A realization of a process of law $\mathfrak{P}^0$ can be obtained by considering a bi-infinite sequence 
$\O=(\O_M)_{M\in \bZ}$ of perfectly matched blocks of i.i.d customers and servers.  Let us perform the exchange transformation and time-reversal on the sequence $\O=(\O_M)_{M\in \bZ}$, and let $\psi\fP^0$ be the probability distribution of the result. Clearly, $\psi\fP$ is the stationary version of $\psi\fP^0$. 
Now according to Lemma \ref{thm.reverse}, we have $\psi\fP^0 =\fP^0$. Since $\psi\fP$ is the stationary version of $\psi\fP^0$ and since $\fP$ is the stationary version of $\fP^0$, we deduce that $\psi\fP=\fP$.

The key point in the argument above is the link between time-stationarity and event-stationarity.  For general background on Palm calculus, see for instance Chapter 1 in \cite{BaBr}. 
So we obtain the following result.

\begin{theorem}
Consider a bipartite matching model under complete resource pooling. 
Let  $(c^n,s^n)_{n\in \bZ}$ be the independent i.i.d. sequences of customers and servers, with the  unique FCFS matching between them.  Then the exchanged sequences $(\tc^n,\ts^n)_{n\in \bZ}$ are independent i.i.d. of the same law as $(c^n,s^n)_{n\in \bZ}$. 
The FCFS matching in reverse time between them (using Loynes' construction in reversed time) consists of the same links as the matching between $(c^n,s^n)_{n\in \bZ}$. 
\end{theorem}

\begin{proof}
That $(\tc^n,\ts^n)_{n\in \bZ},$ are independent i.i.d. sequences follows from the identity of $\psi\fP$ and  $\fP$.  That the Loynes' construction in reversed time will use the same links follows, since the links of $(c^n,s^n)_{n\in \bZ}$ are the FCFS matchings in reversed time between  $(\tc^n,\ts^n)_{n\in \bZ}$, and by Theorem \ref{thm.co-loynes} this matching is unique.
\end{proof}



\section{Stationary Distributions}
\label{sec.stationary}

We have found that for any two independent i.i.d. sequences of customers $C=(c^m)_{m\in\bZ}$ and of servers $S=(s^n)_{n\in\bZ}$,  under complete resource pooling, there is a unique FCFS matching almost surely.  Furthermore, if we exchange every matched pair $(c^m,s^n)$ of customers and servers and retain the matching, we obtain two permuted sequences, of matched and exchanged customers $\tC=(\tc^n)_{n\in\bZ}$, and of matched and exchanged servers $\tS=(\ts^m)_{m\in\bZ}$.  These new sequences are again independent and i.i.d., and the retained matching between them is FCFS in reversed time direction, and it is the unique FCFS matching of $\tC,\,\tS$  in reversed time.

\begin{figure}[htb]
\begin{center}
\includegraphics[scale=0.35]{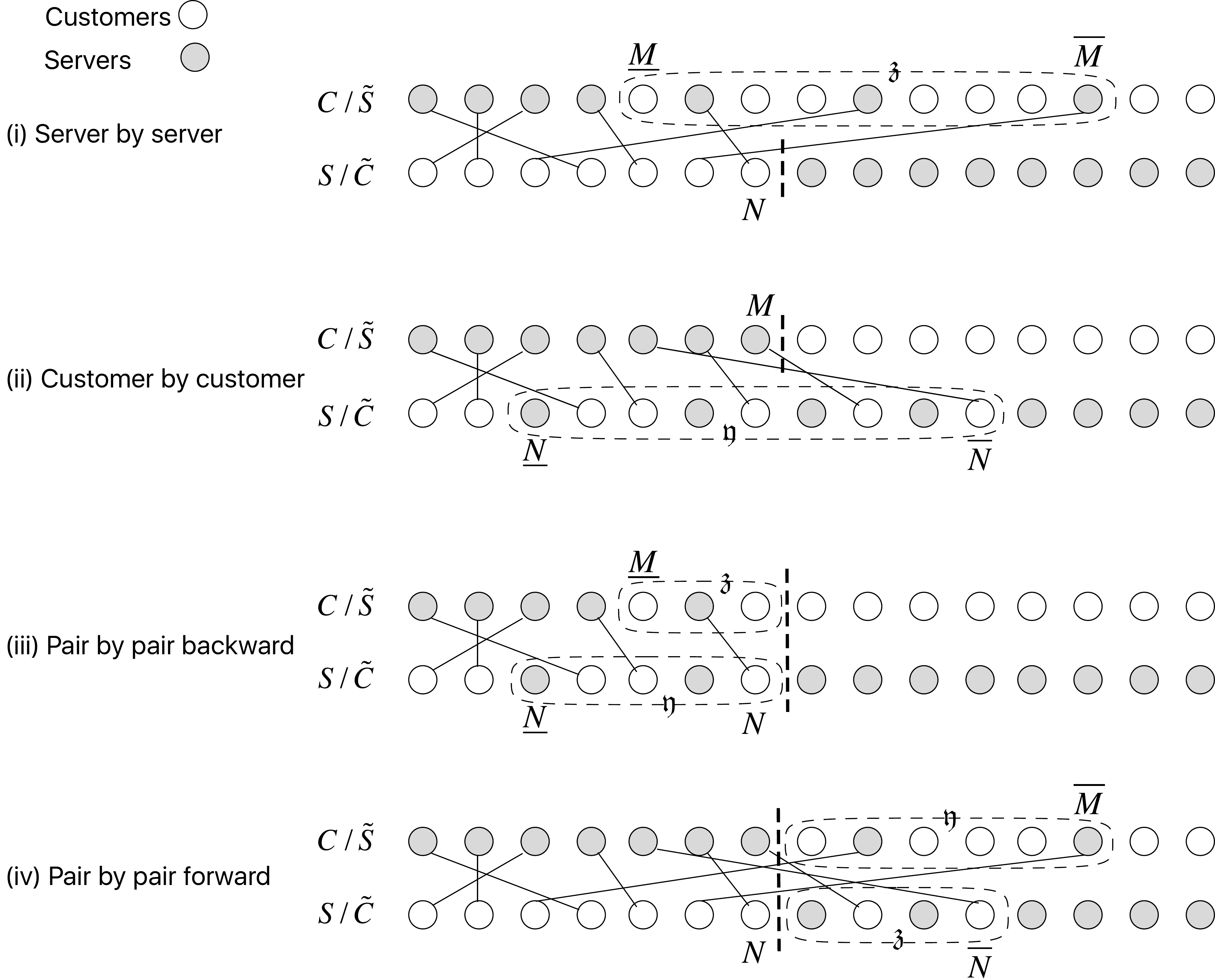}
\end{center}
\caption{Four mechanisms of FCFS matching}
\label{fig.Markov}
\end{figure}

\medskip

In this section we consider the Markovian evolution of the stationary FCFS matching on $\bZ$, and derive stationary distributions of several Markov chains associated with it. 
The FCFS matching of $C,\,S$ evolves moving step by step from the past up to position $N$, where we add matches and perform  exchanges at each step from $N$ to $N+1$.   Four ways in which this can be done are illustrated in Figure \ref{fig.Markov} where light circles represent customers and dark circles represent servers (in original or exchanged positions).   
In each of these, if we reverse the time direction we get a Markovian construction of FCFS matches between $\tC,\, \tS$ that moves from $N+1$ to $N$ and at each of these steps adds matches for elements $\tc^n,\ts^m$ and exchanges them back to $c^m,s^n$.  
We exploit this reversibility to derive the stationary distributions.

The outline of this section is as follows:
In Section \ref{sec.mechanism} we first describe in more details the four mechanisms and define a Markov chain associated with each.  The states of these processes consist of the ordered lists of items in the region encircled by a dashed ellipse in each of the four panels in Figure \ref{fig.Markov}.  We call these the detailed Markov chains.  In Section \ref{sec.reversal} we  formulate  Lemmas \ref{thm.timereversal}, \ref{thm.timereversal2}  on time reversal, that associates each Markov chain in the forward time direction, with a corresponding Markov chain in the reversed time direction.  The proof of these Lemmas is immediate.  

Following this, in Section \ref{sec.distributions},
we use  time reversal to derive  in Theorem \ref{thm.stationary} the stationary distributions of  the detailed Markov chains. These are, up to a normalizing constant, simply the distributions of a finite sequence  of multi-Bernoulli trials.   We also show in Theorem \ref{thm.stationary} that all these distributions possess the same normalizing constant.   

In Section \ref{sec.marginals} we define a Markov chain with an augmented state description, and obtain its stationary distribution as a corollary to Theorem \ref{thm.stationary}.  The advantage of this augmented chain is that its state can be re-interpreted as the state of a queue with parallel servers which is overloaded, as described in \cite{visschers-adan-weiss:12,adan-weiss:11,adan-weiss:14}.  Under this interpretation it is possible to sum over the detailed states and to obtain the stationary distribution of a host of other processes associated with FCFS matching. Furthermore, by summing over all the states we obtain the normalizing constant for the stationary distributions of Theorem \ref{thm.stationary}.  We conjecture that its calculation is $\sharp$-P hard.

Finally, in Section \ref{sec.natural} we again sum over states to obtain the stationary distribution of the `natural' Markov chains.  We illustrate this in Section \ref{sec.example} for the FCFS matching model of the ``NN''-system,
that was only partly analyzed in \cite{caldentey-kaplan-weiss:09}, because its analysis could not be completed at the time.


\subsection{Mechanisms for evolution of FCFS matching and detailed Markov chains}
\label{sec.mechanism}
We consider four mechanisms for the Markovian evolution of the stationary FCFS matching, and define an associated Markov chain for each.
\subsubsection*{Server by server matching}
At time $N$ all servers $s^n,\,n\le N$ have been matched and  exchanged with the customers to which they were matched, as illustrated in panel (i) of Figure \ref{fig.Markov}.  At this point the  servers line has $\tc^n,\,n\le N$ customers that matched and replaced servers $s^n,\,n\le N$, and servers $s^n,\,n>N$ are still unmatched.  On the customers line there is a position $\uM$ such that all the customers $c^m,\,m<\uM$ have been matched and replaced by $\ts^m$, and $c^{\uM}$ is the first unmatched customer, and there is a position $\oM$ such that all customers $c^m,\,m>\oM$ have not yet been matched, and $c^\oM$ is the last customer that has been matched, so that now $\ts^\oM$ is the matched and exchanged server in position $\oM$.  
If the matching for $n\le N$ is perfect then $\oM=\uM-1=N$, otherwise $L=\oM-\uM+1\ge 2$.  We let $\fz=0$ in the former case (sometimes we write $\fz=\emptyset$), and in the latter case we let $\fz=(z^1,\ldots,z^L)$  be the ordered sequence of unmatched customers and of matched and exchanged servers so that $\fz^1=c^\uM$, $\fz^L=\ts^\oM$ and $z^l,\,1<l<L$ is either $c^{\uM+l-1}$ if unmatched or $\ts^{\uM+l-1}$ if matched and exchanged.   

We define the {\em server by server FCFS detailed matching process} $Z^s=(Z^s_N)_{N\in \bZ}$ with $Z^z_N=\fz$.  It is a Markov chain where the transition from $Z^s_N$ to $Z^s_{N+1}$ depends on the current state $\fz$, and on the innovation variables which are the types of server $s^{N+1}$ and of customers $c^m,\,m>\oM$. 

\subsubsection*{Customer by customer matching}
Similar to server by server matching, at time $M$  all customers $c^m,\,m\le M$ have been matched and exchanged with servers, as illustrated in panel (ii) of Figure \ref{fig.Markov}. 
We  define a {\em customer by customer FCFS detailed matching process}, $Z^c=(Z^c_M)_{M\in \bZ}$ 
 so that the state $Z^c_M=\fy$ is $\fy=0$ for perfect match, and otherwise $\fy=(z^1,\ldots,z^L)$ where $z^1=s^{\uN}$ is the first unmatched server on the servers line,  $z^L=\tc^\oN$ is the last matched and exchanged customer, and $z^l,\,1<l<L$  is either $s^{\uN+l-1}$ if unmatched or $\tc^{\uN+l-1}$ if matched and exchanged.  It is a Markov chain where the transition from $Z^c_M$ to $Z^c_{M+1}$ depends on the current state $\fy$, and on the innovation variables which are the types of customer $c^{M+1}$ and of servers $s^n,\,n>\oN$. 

\subsubsection*{Pair by pair backward matching}
For pair by pair backward FCFS matching (illustrated in panel (iii) of Figure \ref{fig.Markov}) we assume that all possible FCFS matches between $s^n,c^m,\,m,n \le N$ have been made and exchanged, and in step $N+1$ we add the pair $s^{N+1},c^{N+1}$, and if possible match and exchange each of them FCFS  to previous unmatched items or to each other.  

We define the pair by pair backwards detailed FCFS matching process $D=(D_N)_{N\in \bZ}$ as $D_N=(\fz,\fy)$, where $\fz=(z^1,\ldots,z^L)$ describes the customers line and $\fy=(y^1,\ldots,y^K)$ describes the server line. Here $z^1$ is the first unmatched customer, in position $N-L+1$, and the remaining items of $\fz$ are either unmatched customers or matched and exchanged servers, $y^1$ is the first unmatched server, in position $N-K+1$, and the  remaining items of $\fy$ are either unmatched servers or matched and exchanged customers.  The number of unmatched customers in $\fz$ needs to be equal to the number of unmatched servers in $\fy$.  We may have $\fz=\fy=0$ if there is a perfect match, otherwise both $L\ge 1$ and $K\ge 1$.  This is a Markov chain, whose next state depends on the current state and the random innovation consists of the types of $s^{N+1},c^{N+1}$.

\subsubsection*{Pair by pair forward matching}
For pair by pair forward FCFS matching (illustrated in panel (iv) of Figure \ref{fig.Markov}) we assume all customers $s^n,c^m,\,m,n \le N$ have been matched and exchanged.  After  step $N$ we consider the pair in position $N+1$, which may contain either items which were matched and exchanged already, or items which are still unmatched, and then in step $N+1$ the items which are still unmatched after step $N$ are  matched and exchanged with each other or with items in positions $> N+1$.

We define the pair by pair forward FCFS matching process $E=(E_N)_{N\in T}$ as $E_N = (\fy,\fz)$, where  $\fy=(y^1,\ldots,y^K)$ lists items in positions $N+1,\ldots,N+K$ on the customer line, where $y^K$ is the last matched and exchanged server $\ts^{N+K}$,  and $y^k,\,1\le k<K$ is either an unmatched customer or a matched and exchanged server in position $N+k$, and where $(z^1,\ldots,z^L)$ lists items in positions $N+1,\ldots,N+L$ on the server line, where $z^L$ is the last matched and exchanged customer $\tc^{N+L}$ and $z^l,\,1\le l<L$ is either an unmatched server, or a matched and exchanged customer in position $N+k$.  $E_N=\emptyset$ after a perfect match, otherwise $K,L \ge 1$. 
$E_N$ is a Markov chain, whose next state depends on the current state, and the random innovation consists of the $c^m,\,m>N+L$ and $s^n,\,n>N+K$.  

It is useful at this point to write down some of the transition probabilities for the detailed Markov chains.  We start with three types of transitions for $Z^s$:
\begin{description}
\item
If $Z_N^s = \fz =( c_{i_0},\ldots,c_i,\ldots,\ts_{j_0})$ and $s^{N+1}=s_j$, such that $c_i$ is the first customer compatible with $s_j$, then $Z_{N+1} = \fz' = (c_{i_0},\ldots,\ts_j,\ldots,\ts_{j_0})$, 
and the probability of this transition is $P(\fz\to \fz') = \beta_{s_j}$.
\item
If $Z_N^s = \fz = (c_i,\ts_{j_1},\ldots,c_{i_0},\ldots,\ts_{j_0})$ and $s^{N+1}$ is compatible with $c_i$, and $c_{i_0}$ is the first unmatched customer following $c_i$, 
then $Z_{N+1} = \fz' = (c_{i_0},\ldots,\ts_{j_0})$, 
and the probability of this transition is $P(\fz\to \fz') = \beta_{\S(c_i)}$.
\item
If $Z_N^s = \fz = (c_{i_0},\ldots,\ts_{j_0})$ and $s^{N+1}=s_j$, and $s_j$ is incompatible with all the unmatched servers in $\fz$, then we may have 
$Z_{N+1}^s = \fz' = (c_{i_0},\ldots,\ts_{j_0},c_{i_1},\ldots,c_{i_k},\ts_j)$, where $c_{i_1},\ldots,c_{i_k}$ are incompatible with $s_j$, with probability: 
$P(\fz\to \fz') = \beta_{s_j} \alpha_{c_{i_1}},\ldots,\alpha_{c_{i_k}} \alpha_{\C(s_j)}$.
\end{description}
The first type of transition of $Z^s$ is a simple exchange of elements, in the second type some elements $z_1,\ldots$ are eliminated, and this may also be a transition to the empty state, and in the third type of transition elements are added after $z_L$.

Transitions for $Z^c$  are similar to those of $Z^s$.

Transitions of the pair by pair process $D$ are  of more types,  we list just two of them, which represent most of the possibilities:
 \begin{description}
\item
If $D_N = (\fz,\fy) = \big((c_{i_0},\ldots,c_i,\ldots),(s_{j_0},\ldots )\big)$, 
and $s^{N+1}=s_k,c^{N+1}=c_l$,  and if $s_k$ is compatible with $c_i$ but not with any earlier unmatched customer in $\fz$, while $c_l$ does not match any of the unmatched servers in $\fy$ then $D_{N+1}=(\fz',\fy') = \big((c_{i_0},\ldots,\ts_k,\ldots,c_l),(s_{j_0},\ldots,\tc_i )\big)$, and the transition probability for this is $P((\fz,\fy)\to (\fz',\fy')) = \alpha_{c_l}\beta_{s_k}$.
\item
If $D_N = (\fz,\fy) = \big((c_{i},\ldots,c_{i_0},\ldots),(s_j,\ldots,s_{j_0},\ldots )\big)$, and $s^{N+1}$ is compatible with $c_i$, and   $c^{N+1}$ is compatible with $s_j$, and the next unmatched customer and server are $c_{i_0},s_{j_0}$, then $D_{N+1}=(\fz',\fy') = \big((c_{i_0},\ldots,\ts_j),(s_{j_0},\ldots,\tc_i )\big)$, and the probability of this transition is 
$P((\fz,\fy)\to (\fz',\fy')) = \alpha_{\C(s_j)} \beta_{\S(c_i)}$.
\end{description}
We see here that in each transition of $D$ a single element is added after $z_K$ and after $y_L$, and sometimes elements $z_1,\ldots$ and/or $y_1,\ldots$ are eliminated.  In the latter case the transition may be to the empty state. 

Transitions of the pair by pair process $E$ are also of several types.  Some of those are similar to transitions of $D$, others are somewhat different. We again list just a few of them, which represent most of the possibilities:
 \begin{description}
\item
If $E_N = (\fy,\fz) = \big((\ts_l,y_1,\ldots,c_i,\ldots,y_k),(s_j,z_1,\ldots,z_l)\big)$, so that $c^{N+1}$ has already been matched and exchanged with $\ts_l$ which was in an earlier position, but $s_j$ has not yet been matched, and $c_i$ is the first compatible unmatched customer in 
$\fy$,  then $E_{N+1}= (\fy',\fz') = \big((y_1,\ldots,\ts_j,\ldots,y_k)$, $(z_1,\ldots,z_l)\big)$, in words, $\ts_j$ replaces $c_i$, and the first elements in $(\fy,\fz)$, are erased.  This transition is deterministic, and happens with probability $P((\fy,\fz)\to (\fy',\fz')) = 1$.
\item
If $E_N = (\fy,\fz) = \big((\ts_l,y_1,\ldots,y_k),(s_j,z_1,\ldots,z_l)\big)$, so that $c^{N+1}$ has already been matched and exchanged with $\ts_l$ which was in an earlier position, but $s_j$ has not yet been matched, and if none of the unmatched customers in $\fy$ are compatible with $s_j$, then $E_{N+1}= (\fy',\fz') = \big((y_1,\ldots,y_k,c_{i_1},\ldots,c_{i_m},\ts_j),(z_1,\ldots,z_l)\big)$, where $s_j$ skipped all of $\fy$ as well as $c_{i_1},\ldots,c_{i_m}$, before finding a match.  The probability of this transition is $P((\fy,\fz)\to (\fy',\fz')) = \alpha_{c_{i_1}}\cdots\alpha_{c_{i_m}}\alpha_{\C(s_j)}$.
\item
If $E_N = \emptyset$, then the next state can be  $E_{N+1} = \emptyset$.  This happens if and only if the next pair, $c^{N+1}=c_i,\,s^{N+1}=s_j$ are compatible.  The probability of this transition is $P((\fy,\fz)\to (\fy',\fz')) = \sum_{(c_i,s_j)\in G} \alpha_{c_i}\beta_{s_j}$
\item
If $E_N = \emptyset$, then the next state can be of the form:  $E_{N+1} = (\fy',\fz') =\big((c_{i_1},\ldots,c_{i_{k-1}},\ts_j)$, $(s_{i_1},\ldots,s_{i_{l-1},}\tc_i)\big)$, $k,l\ge 1$,
where in the transition from $N$ to $N+1$, the new pair is $c^{N+1}=c_i,\,s^{N+1}=s_j$ which are incompatible. In matching them, they are matched and exchanged with some $\ts_{i_l},\tc_{i_k}$, after skipping incompatible  $s_{i_1},\ldots,s_{i_{l-1}}$ and $c_{i_1},\ldots,c_{i_{k-1}}$.  The probability of this transition is: $P((\fy,\fz)\to (\fy',\fz')) =\alpha_{c_i} \beta_{s_j} \beta_{s_{i_1}}\cdots\beta_{s_{i_{l-1}}} \alpha_{c_{i_1}}\cdots\alpha_{c_{i_{k-1}}} \alpha_{\C(s_j)} \beta_{\S(c_i)}$.
\end{description}
We see here that in each transition of $E$, the first element of the state is deleted, and there may be an exchange of elements from $\fy,\fz$, or a geometric number of new elements is added.  Exceptional is the empty state, where transition can be to the empty state, or to a state with a geometric number of new elements in $\fy',\fz'$.


\subsection{Time reversal of the detailed Markov chains}
\label{sec.reversal}

Examining panel (i) of Figure \ref{fig.Markov} we see that it illustrates  FCFS matching and exchange of server and customer lines $C$ and $S$ all the way from $-\infty$ up to position $N$, and at the same time it also illustrates matching and exchange of server and customer lines $\tC$ and $\tS$ FCFS in reversed time, all the way from $\infty$ to $N+1$.    
Our main observation now is that if $Z^s_N=\fz=(z^1,\ldots,z^L)$, then $(z^L,\ldots,z^1)$ is exactly the state of the customer by customer FCFS matching in reversed time  of the sequences $\tC,\,\tS$, when all customers $\tc^n,n\ge N+1$ have been matched to some $\ts^m$, and exchanged back to a customer $c^m$ and server $s^n$.   For $\fz=(z^1,\ldots,z^L)$ we denote $\overleftarrow{\fz}=(z^L.\ldots,z^1)$.   We denote the Markov chain of customer by customer FCFS matching 
of $\tC,\,\tS$ in reversed time by $\overleftarrow{Z}$, so that $\overleftarrow{Z}^c_N$, is the state where all $\tc^n,\,n > N$ have been matched.  
We then state formally (see Figure \ref{fig.reversal1}):
\begin{lemma}[Time reversal]
\label{thm.timereversal}
The Markov chain $\overleftarrow{Z}^c_N$ of customer by customer FCFS matching of $\tC,\,\tS$ in reversed time, is the time reversal of the Markov chain $Z^s_N$ of server by server FCFS matching of $C,\,S$, in the sense that
\begin{equation}
\label{eqn.reversal}
Z^s_N = \fz, \; Z^s_{N+1} = \fz'  \mbox{ if and only if } 
\overleftarrow{Z}^c_{N+1} = \overleftarrow{\fz'}, \;    \overleftarrow{Z}^c_N = \overleftarrow{\fz}.
\end{equation}
This implies that the reversal of the transition  $Z^s_N=\fz \to Z^s_{N+1}=\fz'$ is exactly the  transition
$\overleftarrow{Z}^c_{N+1}=\overleftarrow{\fz'} \to \overleftarrow{Z}^c_N=\overleftarrow{\fz}$.  In other words, if the transition of $Z^s_N \to Z^s_{N+1}$ matches and exchanges $s^n$ with $c^m$, then 
the transition of $\overleftarrow{Z}^c_{N+1} \to \overleftarrow{Z}^c_N$ matches and exchanges $\tc^n$ with $\ts^m$.
\end{lemma}
 \begin{figure}[htb]
\begin{center}
\includegraphics[scale=0.30]{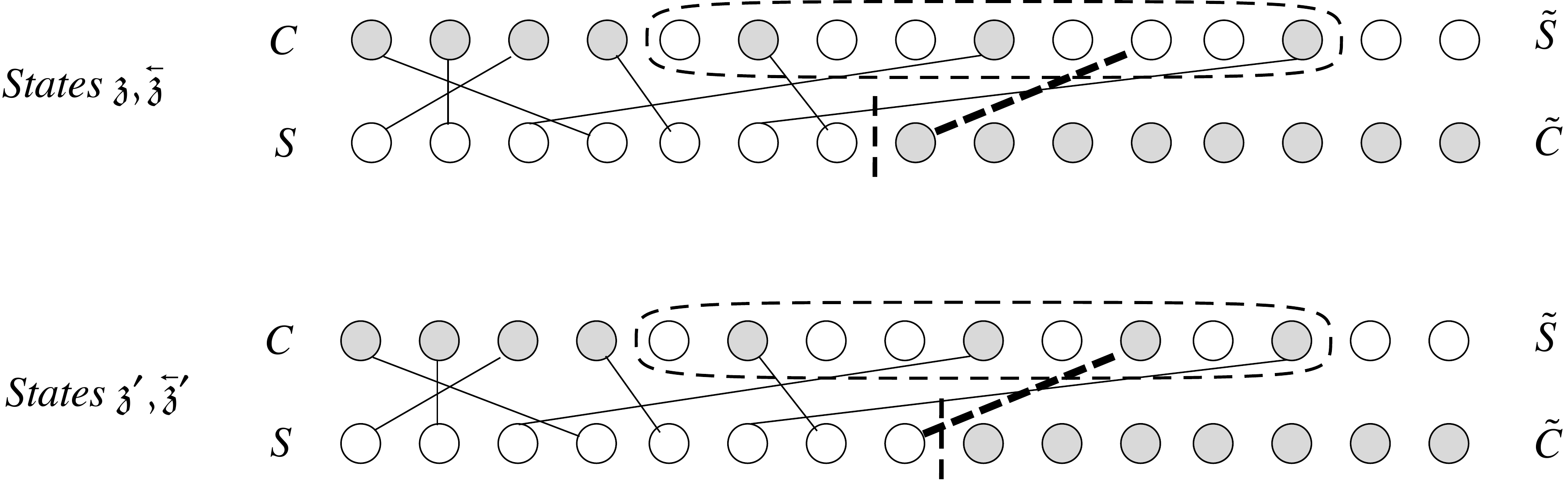}
\end{center}
\caption{Single match and exchange of customer server pair and its reversal}
\label{fig.reversal1}
\end{figure}
\begin{proof}
The only thing to note is that if $s^n,\,c^m$ are matched in the forward FCFS direction, then 
$\ts^m,\,\tc^n$ are matched in the reversed FCFS direction.  This is so by the almost sure uniqueness of the matching as shown in Theorem \ref{thm.co-loynes}, and by Lemma \ref{thm.reverse}.
\end{proof}

A similar observation on time reversal holds also for the pair by pair backward and forward detailed Markov chains.
Examining panel (iii) of Figure \ref{fig.Markov} we again see that it illustrates  FCFS matching and exchange of server and customer lines $C$ and $S$ all the way from $-\infty$ up to position $N$, and at the same time it also illustrates matching and exchange of server and customer lines $\tC$ and $\tS$ FCFS in reversed time, all the way from $\infty$ to $N+1$.    
Our main observation now is that if $Z^s_N=(\fz,\fy)=\big((z^1,\ldots,z^L),(y^1,\ldots,y^K)\big)$, then $\big((y^K,\ldots,y^1),(z^L,\ldots,z^1)\big)=(\overleftarrow{\fy},\overleftarrow{\fz})$ is exactly the state of the pair by pair forward detailed  FCFS matching in reversed time,   of the sequences $\tC,\,\tS$, when all customers and servers $\tc^n,\ts^m,\,m,n > N$ have been matched  and exchanged back to a customer $c^m$ and server $s^n$.     We denote the pair by pair forward detailed FCFS matching 
of $\tC,\,\tS$ in reversed time by $\overleftarrow{E}_N$.  We then state formally (see Figure \ref{fig.reversal2}):

\begin{lemma}[Time reversal]
\label{thm.timereversal2}
The Markov chain $\overleftarrow{E}_N$ of pair be pair forward FCFS matching of $\tC,\,\tS$ in reversed time, is the time reversal of the Markov chain $D_N$ of pair by pair backward FCFS matching of $C,\,S$, in the sense that
\begin{equation}
\label{eqn.reversal3}
D_N = (\fz,\fy), \; D_{N+1} = (\fz',\fy')   \mbox{ if and only if } 
\overleftarrow{E}_{N+1} = (\overleftarrow{\fy'},\overleftarrow{\fz'}), \;
\overleftarrow{E}_N = (\overleftarrow{\fy},\overleftarrow{\fz}).
\end{equation}
This implies that the reversal of the transition  $D_N=\fz \to Z_{N+1}=\fz'$ is exactly the  transition
$\overleftarrow{E}_{N+1}=(\overleftarrow{\fy'},\overleftarrow{\fz'}) \to \overleftarrow{E}_N=(\overleftarrow{\fy},\overleftarrow{\fz})$.  In other words, if the transition of $D_N \to D_{N+1}$ looks for matches for $c^{N+1},s^{N+1}$ and exchanges $s^n$ with $c^m$, then 
the transition of $\overleftarrow{E}_{N+1} \to \overleftarrow{E}_N$ considers the elements in position $N$, and transforms them back.
\end{lemma}
 \begin{figure}[htb]
\begin{center}
\includegraphics[scale=0.30]{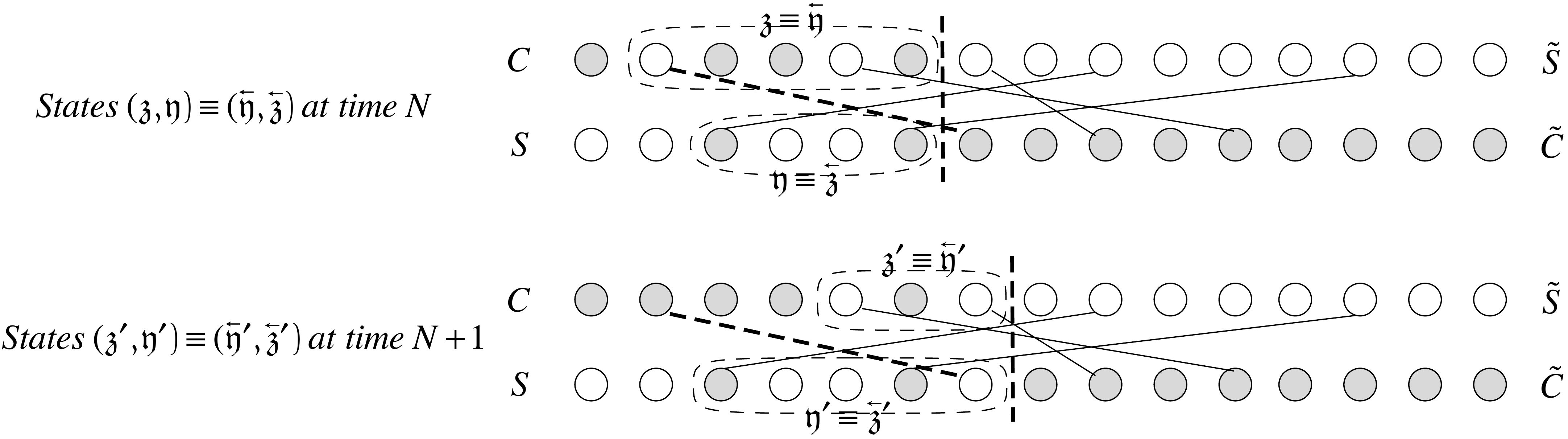}
\end{center}
\caption{Adding pair backward and its reversal adding pair forward}
\label{fig.reversal2}
\end{figure}
\begin{proof}
The only thing to note, as in the proof of Lemma \ref{fig.reversal1}, is that if $s^n,\,c^m$ are matched in the forward FCFS direction, then 
$\ts^m,\,\tc^n$ are matched in the reversed FCFS direction.
\end{proof}


\subsection{Stationary distributions of the detailed Markov chains}
\label{sec.distributions}
We are now ready to derive the stationary distributions of the detailed Markov chains.

\begin{theorem}
\label{thm.stationary}
(i) The stationary distribution of $Z^s_N$ and of $Z^c_M$ is given, up to a normalizing constant, by
\begin{equation}
\label{eqn.stationrydist}
\pi_{Z^s}(\fz)  = \pi_{Z^c}(\overleftarrow{\fz})   = \prod_{i=1}^{I}  {\alpha_{c_i}}^{\# c_i} \prod_{j=1}^{J} {\beta_{s_j}}^{\# s_j}, \qquad
\end{equation}
where $\# c_i$ is the number of customers of type $c_i$,
and  $\# s_j$ is the number of servers of type $s_j$, as they appear in the state $\fz$.   

(ii)  The stationary distribution of $D_N$ and of $E_N$ is given, up to a normalizing constant, by
\begin{equation}
\label{eqn.pairdist}
\pi_D(\fz,\fy) = \pi_E(\overleftarrow{\fy},\overleftarrow{\fz}) =  \prod_{i=1}^{I}  {\alpha_{c_i}}^{\# c_i}  \prod_{j=1}^{J} {\beta_{s_j}}^{\# s_j}, 
\end{equation}
where $\# c_i$,  $\# s_j$ count customers and servers as they appear in the state $(\fz,\fy)$. 

(iii)  The normalizing constant is the same in all four distributions.
\end{theorem}
\begin{proof}
We prove (i) using Kelly's Lemma, and the time reversal result.  The proof of (ii) is similar and in included in Appendix \ref{sec.appendix} for completeness.  To show (iii) we show that there is a 1-1 correspondence between states of $Z^s_N$ and $D_N$.

{\em Proof of (i)}

We use Kelly's Lemma (\cite{kelly:79}, Section 1.7):  For a Markov chain $X_t$, if we can find non-negative $\pi(i)$ and   $p_{j\to i}$ such that 
\begin{equation}
\label{eqn.conditions}
\sum_i p_{j\to i} =1 \mbox{ for all $j$}, \; \mbox{ and }\quad 
\pi(i) P(X_{t+1}=j\,|\,X_t=i) = \pi(j) p_{j\to i} \mbox{ for all $i,j$ } 
\end{equation}
then $\pi$ is the stationary distribution of $X_t$, and $p_{j\to i}$ are the transition rates of the reversed stationary process, $p_{j\to i}= P(X_t=i\,|\,X_{t+1}=j)$.

Lemma \ref{thm.timereversal} shows that the  time reversed transition probabilities of $Z^s_N$ are
\begin{equation}
\label{eqn.reversal2}
P(Z^s_N=\fz \,|\,Z^s_{N+1}=\fz')  = P(\overleftarrow{Z}^c_N=\overleftarrow{\fz} \,|\, \overleftarrow{Z}^c_{N+1}=\overleftarrow{\fz'}) = P(Z^c_{M+1} =\overleftarrow{\fz} \,|\, Z^c_M =\overleftarrow{\fz'}).
\end{equation}

The  condition of summation to 1 in (\ref{eqn.conditions}) is satisfied since $P(Z^c_{M+1} =\fy' \,|\, Z^c_M =\fy )$ are transition probabilities. 

To check the  condition of detailed balance in (\ref{eqn.conditions}) all we need to do is check that the conjectured form of $\pi_{Z^s}$ in (\ref{eqn.stationrydist})  satisfies 
\begin{equation}
\label{eqn.checking1}
\pi_{Z^s}(\fz) P(Z^s_{N+1}=\fz' \,|\,Z^s_N=\fz) = 
\pi_{Z^s}(\fz') P(Z^s_n=\fz \,|\,Z^s_{N+1}=\fz') = 
\pi_{Z^c}(\overleftarrow{\fz'}) 
P(Z^c_{M+1} =\overleftarrow{\fz} \,|\, Z^c_M =\overleftarrow{\fz'}).
\end{equation}
There are  three types of transitions of $Z^s$ (see Figure \ref{fig.serverbyserver}), with analogous transitions for $Z^c$ 
 \begin{figure}[htb]
\begin{center}
\includegraphics[scale=0.30]{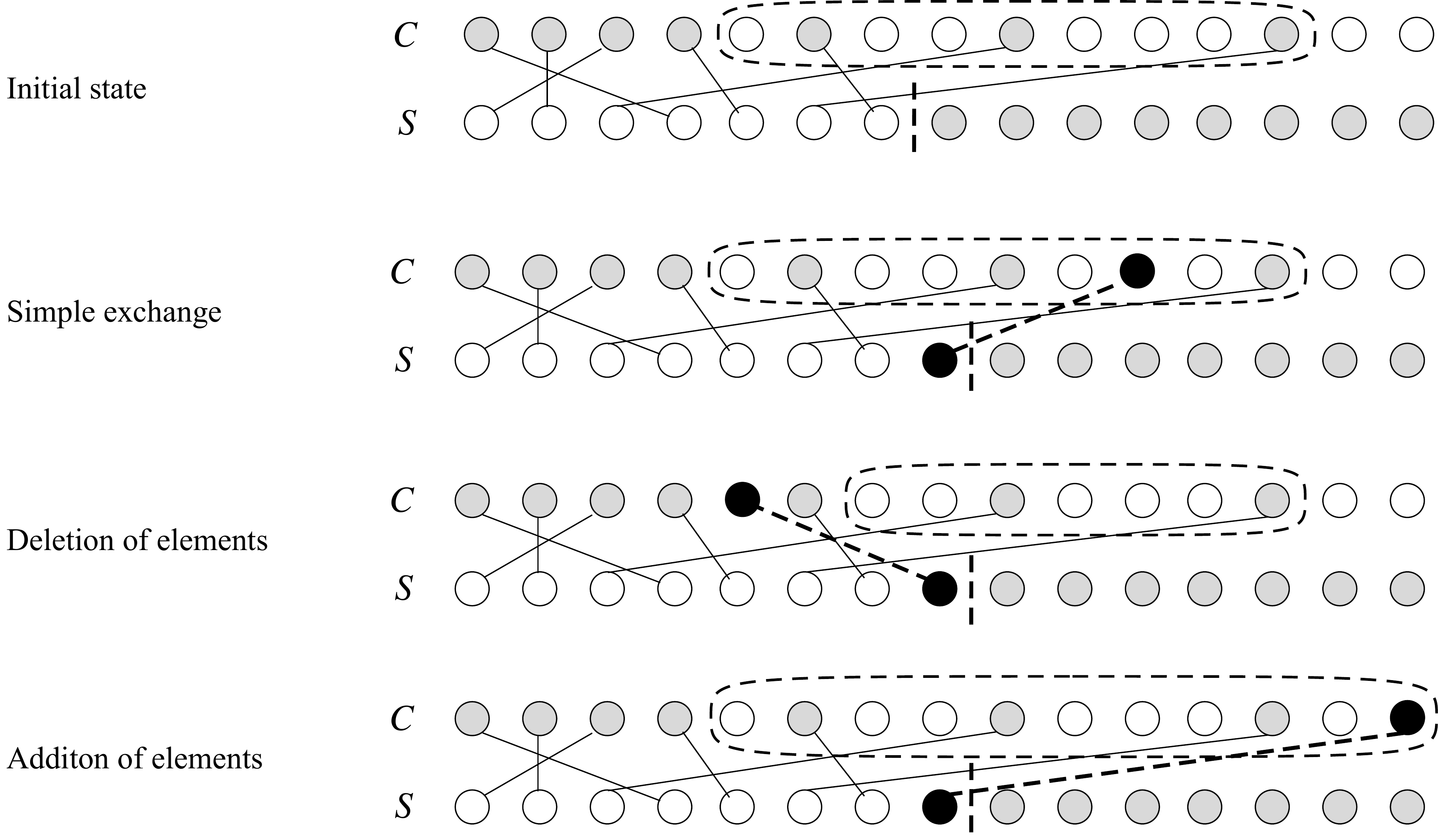}
\end{center}
\caption{Possible transitions of the server by server FCFS Markov chain $Z^s_N$}
\label{fig.serverbyserver}
\end{figure}
We check (\ref{eqn.checking1}) for each of them.

{\em (a) Simple exchange:}  If $z^l = c_i$ and ${z'}^l = s_j$, 
then according to  (\ref{eqn.stationrydist}) we have
\[
\pi_{Z^c}(\overleftarrow{\fz'}) = \pi_{Z^s}(\fz') =  \pi_{Z^s}(\fz)  \frac{\beta_{s_j}}{\alpha_{c_i}} 
\]
and the direct and reversed transition rates are:
\[
P(Z^s_{N+1} = \fz' | Z^s_N = \fz) = \beta_{s_j}, \qquad 
P(Z^c_{M+1} = \overleftarrow{\fz} | Z^c_M = \overleftarrow{\fz'}) = \alpha_{c_i}.
\]

{\em (b) Deletion of elements from start of $\fz$}:  If $\fz = z^1,\ldots,z^L$, and $\fz' = z^{k+1},\ldots,z^L$, and the deleted part is $c_i, s_{j_1}, \ldots, s_{j_{k-1}}$, then according to (\ref{eqn.stationrydist}) we have
\[
\pi_{Z^s}(\fz') =  \pi_{Z^s}(\fz) \frac{1}{\alpha_{c_i} \beta_{s_{j_1}} \cdots \beta_{s_{j_{k-1}}}}
\]
and the direct and reversed transition rates are:
\[
P(Z^s_{N+1} = \fz' | Z^s_N = \fz) = \beta_{\S(c_i)},  \qquad
P(Z^c_{M+1} = \overleftarrow{\fz} | Z^c_M = \overleftarrow{\fz'}) = 
\alpha_{c_i} \beta_{s_{j_1}} \cdots \beta_{s_{j_{k-1}}}  \beta_{\S(c_i)}.
\]
 
 {\em (c) Addition of elements to $\fz$}:   If $\fz=z^1,\ldots,z^l$ and $\fz'=z^1,\ldots,z^l,c_{i_1},\ldots,c_{i_k},s_j$,   then 
 according to (\ref{eqn.stationrydist}) we have
\[
\pi_{Z^s}(\fz') =  \pi_{Z^s}(\fz) \alpha_{c_{i_1}} \cdots \alpha_{c_{i_k}} \beta_{s_j}
\]
and the direct and reversed transition rates are:
\[
P(Z^s_{N+1} = \fz' | Z^s_N = \fz) = \beta_{s_j} \alpha_{c_{i_1}} \cdots \alpha_{c_{i_k}} 
\alpha_{\C(s_j)} ,  \qquad
P(Z^c_{M+1} = \overleftarrow{\fz} | Z^c_M = \overleftarrow{\fz'}) = \alpha_{\C(s_j)} .
\]
It is now immediate to check that  the balance condition of (\ref{eqn.conditions}) holds.

This also proves the form of  $\pi_{Z^c}$, as the reversed process.   

\medskip
{\em Proof of (ii)}  This is similar, and we included it in Appendix \ref{sec.appendix}

\medskip
{\em Proof of (iii)}
%

We show that there is a one to one correspondence between the states of $Z^s_N$ and of $D_N$, such that the stationary probabilities of corresponding states are equal up to the normalizing constants.  This implies equality of the normalizing constants of $\pi_D$ and $\pi_{Z^s}$.  The other processes have the same normalizing constants by the time reversibility.

We note first that a sequence of servers and customers $(z^1,\ldots,z^L)$ is a possible state $\fz$ of $Z^s$ if and only if $z^l=c_i,z^k=s_j,l<k$ implies $(c_i,s_j)\not\in\E$, and similarly it is a possible state $\fy$ of $Z^c$ if and only if $z^l=c_i,z^k=s_j,l > k$ implies $(c_i,s_j)\not\in\E$.  Necessity follows from FCFS, sufficiency follows by constructing sequences  $(c^1,\ldots,c^L)$ and $(s^1,\ldots,s^L)$ whose matching yields $(z^1,\ldots,z^l)$ on the customer line for $\fz$, or on the server line for $\fy$.

To show the correspondence consider a state $D_N=(\fz,\fy)=\big((z^1,\ldots,z^L),(y^1,\ldots,y^K)\big)$.  Then  $\fz'=(z^1,\ldots,z^L,y^K,\ldots,y^1)$ is a possible state of $Z^s$.  We need to show  that any customer (unmatched, or matched and exchanged) in the sequence 
$(z^1,\ldots,z^L,y^K,\ldots,y^1)$ is incompatible with any server (unmatched, or matched and exchanged) that appears later in the sequence.  Consider $z^l=c^l$, i.e. $c^l$ is unmatched.  Then it is incompatible with any $\ts^j=z^j,\,l<j\le L$, because it was skipped by them in the matching of the customer line, and it is incompatible with any of $y^k=s^k$ because all unmatched customers in $\fz$ are incompatible with all unmatched servers in $\fy$.    Consider then $y^k=\tc^k$.  It is incompatible with any $s^j=y^j,\,1\le j <k$, because it was skipped by  $\tc^k$ in the matching of the server line.  This shows the correspondence in one direction.

Consider now a state of $Z^s$, $\fz=(z^1,\ldots,z^K)$.   If we break it into $(z^1,\ldots,z^L)$ and $(y^{K-L}=z^{L+1},\ldots,y^1=z^K)$, we get two subsequences, $\fz'=(z^1,\ldots,z^L)$ and $\fy=(y^1,\ldots,y^{K-L})$.  By the above arguments, any matched and exchanged $\ts$ in $\fz$ cannot be matched to any earlier unmatched customer in the sequence $\fz'$ or to any customer in $\fy$, and any matched and exchanged $\tc$ in $\fy$ cannot be matched to any earlier unmatched server in $\fy$ or to any matched and exchanged server in $\fz'$.  So it remains to choose $L$ in such a way  that the number of unmatched customers in $\fz'$ is equal to  the number of unmatched servers in $\fy$.  We do that as follows: we count the same number of customers $c^l$ from the left of $\fz$ as of servers $\ts^k$ from the right of $\fz$, for as long as $l<k$.  The final pair will either be $l=k-1$ in which case we take $L=l$, or $l < k-1$.  In that case all of $z^{l+1},\ldots,z^{k-1}$ must all be either servers, in which case we let $L=k-1$, or all of them are customers, in which case we let $L=l$.  This shows the correspondence in the other direction.

We note that two corresponding states have the same set of customers and the same set of servers.  Hence their steady state probabilities are equal up to the normalizing constant.  This shows that they have the same normalizing constant, and completes the proof.

\end{proof}


\subsection{Augmented state, marginals and the normalizing constant}
\label{sec.marginals}

The extreme simplicity of the stationary probabilities obtained in Theorem \ref{thm.stationary} is deceptive, since it does not indicate which states are possible, according to the compatibility graph and the FCFS matching policy.  In particular, there seems to be no simple way of deciding what are all the possible states of the four detailed Markov chains.  As a result it is not at all obvious how to calculate the normalizing constant for the distributions (\ref{eqn.stationrydist}), (\ref{eqn.pairdist}).  
To allow us to classify states in a convenient way, and thus to allow us to count them and to add up their stationary probabilities, we define an augmented detailed Markov chain.  It also describes the server by server FCFS matching mechanism, but its states are augmentations of the states of  $Z^s$, in that they are even more detailed.

Consider server by server FCFS matching when all servers up to position $N$ have been matched and exchanged.  For the server by server FCFS detailed matching process we defined $Z^s_N$ as the sequence of elements from position $\uM$ of the first unmatched customer to position $\oM$ of the last matched and exchanged server on the customer line.  We now consider positions $\uN < \uM$ and $\oN> \oM$ such that the interval of positions $\uN$ to $\oN$ contains for each server type at least one matched and exchanged server, and it contains for each customer type at least one unmatched  customer, and  the interval is minimal.   Let $\fz=(z^1,\ldots,z^K)$ where $K=\oN-\uN+1$,  $z^1=\ts^\uN$, $z^K=c^\oN$, and for $\uN<l<\oN$, $z^l$ is either an unmatched customer or a matched and exchanged server in position $\uN+l-1$.
We consider the process $\o Z^s=(\o Z^s_N)_{N\in \bZ}$
where $\o Z^s_N=\fz$.  Note that $\o Z^s_N=\fz$ differs from $Z^s_N$ by the addition of some servers before $c^\uM$ and some customers after $\ts^\oM$.  We always have $K\ge I+J$.

We define also the {\em  server by server FCFS augmented matching process} $\Z=(\Z_N)_{N\in T}$ with state $\fz=(z^1,\ldots,z^L)$ with $L=\oM-\uN+1 \ge J$, which includes elements from positions $\uN$ to $\oM$ on the customer line, starting with $z^1=\ts^\uN$ and ending with $z^L=\ts^\oM$
\begin{corollary}
\label{thm.stationary2}
The stationary distributions of $\o Z^s$ and of $\Z$ are given by
\begin{eqnarray}
\label{eqn.stationrydist3}
\pi_{\o Z^s}(\fz)    &=& B  \prod_{i=1}^{I}  {\alpha_{c_i}}^{\# c_i} \prod_{j=1}^{J} {\beta_{s_j}}^{\# s_j}, \\
\label{eqn.stationrydist4}
\pi_{\Z}(\fz)    &=& B  \prod_{i=1}^{I}  {\alpha_{c_i}}^{\# c_i} \prod_{j=1}^{J} {\beta_{s_j}}^{\# s_j}, 
\end{eqnarray}
where $\# c_i$ is the number of customers of type $c_i$ in $\fz$,
and  $\# s_j$ is the number of servers of type $s_j$ in $\fz$.   
\end{corollary}
\begin{proof}
The proof of the form of  $\pi_{\o Z^s}(\fz)$, up to the normalizing constant, is similar to the proof of Theorem \ref{thm.stationary}, since the reversed process corresponds to the customer by customer matching of $\tC,\tS$.   

The proof of the form of $\pi_{\Z}(\fz)$ is obtained by summing $\pi_{\o Z^s}(\cdot)$ over all sequences of customers $c^{\oM+1},\ldots,c^\oN$, which sum up to 1, because we are summing over all outcomes of one or more geometric distributions.  

Finally, the proof that the normalizing constant is again equal to $B$ is obtained by summing also over all sequences $\ts^\uN,\ldots,\ts^{\uM-1}$.  Summation over all added servers and customers returns us to $\pi_{Z^s}(\cdot)$.
\end{proof}

The motivation for considering the augmented process $\Z$ is that each state $\Z_N=\fz$ can be written in a different form, and in that form we can actually enumerate all the possible states.  This enables us to obtain stationary distributions of various marginal processes, and finally to derive an explicit expression for the normalizing constant $B$.  We now rewrite the state $\fz=z^1,\ldots,z^L$ as follows:  Let $S_J$ be the type of server $z^L=\ts^\oM$.  Define recursively, for $1 \le j < J$, $S_j$ as the type of the last server in the sequence $z^1,\ldots,z^L$ which is different from $S_{j+1},\ldots,S_J$.  Then $R=(S_1,\ldots,S_J)$ is a permutation of the server types $s_1,\ldots,s_J$.  Let $\fw_1,\ldots,\fw_{J-1}$ be the subsequences of customer and server types between the locations of $S_1,\ldots,S_J$ in $\fz$.  We will then write the state as $\fz=(S_1,\fw_1,\ldots,\fw_{J-1},S_J)$.   The idea of presenting the state in this form stems from  \cite{visschers:00,visschers-adan-weiss:12} and was used in \cite{adan-weiss:11,adan-weiss:14}.

The main feature of $\fz=(S_1,\fw_1,\ldots,\fw_{J-1},S_J)$ is that all the customers in $\fw_\ell$ are of types in $\U(S_1,\ldots,S_\ell)$ and all the servers in $\fw_\ell$ are of types in $\{S_{\ell+1},\ldots,S_J\}$.
Of course we can write the stationary distribution of states of $\Z$, given in (\ref{eqn.stationrydist4}), also as:
\begin{equation}
\label{eqn.stationaryZ}
\pi_{\Z}(S_1,\fw_1,\ldots,\fw_{J-1},S_J) = B \prod_{j=1}^J \beta_{s_j} 
\prod_{\ell=1}^{J-1} \left(
 \prod_{c_i\in \U\{S_1,\ldots,S_\ell\}} \alpha_{c_i}^{\#(c_i,\fw_\ell)}
\prod_{s_j\in \{S_{\ell+1},\ldots,S_J\}} \beta_{s_j}^{\#(s_j,\fw_\ell)} \right),
\end{equation}
where $\#(c_i,\fw_\ell)$, $\#(s_j,\fw_\ell)$ count the number of type $c_i$ customers and of type $s_j$ servers in $\fw_\ell$.  We will use the notation $B^s=B \prod_{j=1}^J \beta_{s_j} $.

We now consider several processes, some of them Markovian and some of them not Markovian, which are derived by aggregating the states of the detailed augmented Markov chain $\Z$.
\begin{description}
\item
The  process $W_N = (S_1,w_1,\ldots,w_{J-1},S_J)$, where  $w_\ell$ is obtained from $\fw_\ell$ by replacing each $c^m$ with a $0$, and each $\ts^n$ with a $1$, so each $w_\ell$ is a sequence of $0$ and $1$.
\item
The process $X_N=(S_1,n_1,\ldots,n_{J-1},S_J)$ where  $n_\ell = \sharp 0(w_\ell)$ is the number of unmatched customers between $S_\ell$ and $S_{\ell+1}$.
\item
The process $Y_N=(S_1,m_1,\ldots,m_{J-1},S_J)$ where  $m_\ell = \sharp 1(w_\ell)$ is the number of matched and exchanged servers between $S_\ell$ and $S_{\ell+1}$.
\item
The process $U_N=(S_1,n_1,m_1,\ldots,n_{J-1},m_{J-1},S_J)$ which records both the number of   unmatched customers and matched and exchanged servers between $S_\ell$ and $S_{\ell+1}$. 
\item
The process $V_N =( S_1,n_1+m_1,\ldots,n_{J-1}+m_{J-1},S_J)$ which records the lengths of the $w_\ell$.
\item
The process $R_N=(S_1,\ldots,S_N)$ which is the  permutation of server types after the $N$th match.
\end{description}

\begin{theorem}
The stationary distributions of $W,X,Y,U,V,R$ are given by:
\begin{eqnarray}
\label{eqn.stationaryW}
\hspace{-0.4in} &&\pi_W(S_1,w_1,\ldots,w_{J-1},S_J) = B^s \prod_{\ell=1}^{J-1} 
  \alpha_{\U\{S_1,\ldots,S_\ell\}}^{\#0(w_\ell)}
 \beta_{\{S_{\ell+1},\ldots,S_J\}}^{\#1(w_\ell)},  \\
\label{eqn.stationaryU}
\hspace{-0.4in} &&\pi_U(S_1,n_1,m_1,\ldots,S_J)   
=  B^s \prod_{\ell=1}^{J-1}  \binom{n_\ell+m_\ell}{n_\ell}
\left(\alpha_{\U\{S_1,\ldots,S_\ell\}}\right)^{n_\ell}
\left(\beta_{\{S_{\ell+1},\ldots,S_J\}}\right)^{m_\ell}, \\
\label{eqn.stationaryX}
\hspace{-0.4in} &&\pi_X(S_1,n_1,\ldots,n_{J-1},S_J) =
B^s \prod_{\ell=1}^{J-1}  
\frac{  \left(\alpha_{\U\{S_1,\ldots,S_\ell\}}\right)^{n_\ell}  }
{  \left(\beta_{\{S_1,\ldots,S_\ell\}}\right)^{n_\ell+1}  }, \\
\label{eqn.stationaryY}
\hspace{-0.4in} &&\pi_Y(S_1,m_1,\ldots,m_{J-1},S_J)  =
B^s \prod_{\ell=1}^{J-1}  
\frac{  \left(\beta_{\{S_{\ell+1},\ldots,S_J\}}\right)^{m_\ell}  }
{  \left(\alpha_{\C\{S_{\ell+1},\ldots,S_J\}}\right)^{m_\ell+1}  }, \\
\label{eqn.stationaryV}
\hspace{-0.4in} && \pi_V(S_1,r_1,\ldots,r_{J-1},S_J)  =
B^s \prod_{\ell=1}^{J-1}  
\left( \alpha_{\U\{S_1,\ldots,S_\ell\}} +  \beta_{\{S_{\ell+1},\ldots,S_J\}}\right)^{r_\ell}, \\
\label{eqn.stationaryR}
\hspace{-0.4in} && \pi_R(S_1,\ldots,S_J)  = 
B^s \prod_{\ell=1}^{J-1}  
\left(  \beta_{\{S_1,\ldots,S_\ell\}}  - \alpha_{\U\{S_1,\ldots,S_\ell\}} \right)^{-1}. 
\end{eqnarray}
\end{theorem}
\begin{proof}
To obtain (\ref{eqn.stationaryW}) we sum (\ref{eqn.stationaryZ}), where each 0 in $w_\ell$ can be any $c_i\in \U(S_1,\ldots,S_\ell)$, and each 1 in $w_\ell$ can be any $s_j\in \{S_{\ell+1},\ldots,S_J\}$.

To obtain  (\ref{eqn.stationaryU}) we sum (\ref{eqn.stationaryW}) over the different locations of the 0's in each $w_\ell$.

To obtain (\ref{eqn.stationaryX}) we sum (\ref{eqn.stationaryU}) where each $m_\ell$ needs to range from 0 to $\infty$.  We use Newton's binomial formula for $(1-\beta_{\{S_{\ell+1},\ldots,S_J\}})^{-(n_\ell+1)}$, and replace that by $(\beta_{\{S_1,\ldots,S_\ell\}})^{-(n_\ell+1)}$. 

(\ref{eqn.stationaryY}) is obtained similar to (\ref{eqn.stationaryX}) by summing (\ref{eqn.stationaryU}) over the $n_\ell$.

To obtain (\ref{eqn.stationaryV}) we sum (\ref{eqn.stationaryU}), where $n_\ell$ ranges from 0 to $r_\ell=n_\ell+m_\ell$.  We use Newton's binomial formula for   
$\left( \alpha_{\U\{S_1,\ldots,S_\ell\}} +  \beta_{\{S_{\ell+1},\ldots,S_J\}}\right)^{r_\ell}$. 

Finally, to obtain (\ref{eqn.stationaryR}) we sum (\ref{eqn.stationaryV}) over each $r_\ell$ ranging form 0 to $\infty$, and replace $(1-\beta_{\{S_{\ell+1},\ldots,S_J\}})$ by $\beta_{\{S_1,\ldots,S_\ell\}}$.
\end{proof}

We are now ready to obtain the normalizing constant $B$ (see \cite{adan-weiss:11}).  

\begin{theorem}
The normalizing constant $B$ is given by:
\begin{eqnarray}
\label{eqn.B}
B &=&    \left( \prod_{j=1}^J \beta_{s_j}  \times \sum_{(S_1,\ldots,S_J)\in \P_J}   \prod_{\ell=1}^{J-1}  
\left(  \beta_{\{S_1,\ldots,S_\ell\}}  - \alpha_{\U(\{S_1,\ldots,S_\ell\})} \right)^{-1}  \right)^{-1} \\
   &=&   \left(  \prod_{i=1}^I \alpha_{c_i}  \times \sum_{(C_1,\ldots,C_I)\in \P_I}   \prod_{\ell=1}^{I-1}  
\left(  \beta_{\S(\{C_1,\ldots,C_\ell\})}  - \alpha_{\{C_1,\ldots,C_\ell\}} \right)^{-1}   \right)^{-1}.
\end{eqnarray}
where the summation is over all permutations of $s_1,\ldots,s_J$ in the first expression, and over all permutations of $c_1,\ldots,c_I$ in the second expression.
\end{theorem}
\begin{proof}  
We obtain $B^s$ by summation of (\ref{eqn.stationaryR}) over all permutations of $s_1,\ldots,s_J$, and obtain $B$ by dividing it by  $\prod_{j=1}^J \beta_{s_j}$.  The second expression is obtained by using the symmetric derivation for  customer by customer matching.
\end{proof}

By observing when $B$ is finite we obtain:
\begin{corollary}
A necessary and sufficient condition for ergodicity of all the FCFS matching Markov chains is complete resource pooling (\ref{eqn.pooling}).
\end{corollary}

\begin{corollary}
The conditional distributions of the numbers of unmatched customers and of matched and exchanged servers, given the permutation is a product of geometric probabilities:
\begin{eqnarray}
\label{eqn.geometric-c}
\hspace{-0.5in} && P(X_N=(S_1,n_1,\ldots,n_{J-1},S_J)\,|\, S_1,\ldots,S_J)=
 \prod_{\ell=1}^{J-1}  \left(  \frac{  \alpha_{\U\{S_1,\ldots,S_\ell\}}  }
{ \beta_{\{S_1,\ldots,S_\ell\}}  }\right)^{n_\ell}
  \left(  1  - \frac{\alpha_{\U\{S_1,\ldots,S_\ell\}}}{\beta_{\{S_1,\ldots,S_\ell\}}} \right), \\
\label{eqn.geometric-s}
\hspace{-0.5in} && P(Y_N=(S_1,m_1,\ldots,m_{J-1},S_J)\,|\,S_1,\ldots,S_J )   =  \prod_{\ell=1}^{J-1}  \left(  \frac{  \beta_{\{S_{\ell+1},\ldots,S_J\}}  }
{  \alpha_{\C\{S_{\ell+1},\ldots,S_J\}}  }\right)^{m_\ell}
  \left(  1  - \frac{\beta_{\{S_{\ell+1},\ldots,S_J\}}}{\alpha_{\C\{S_{\ell+1},\ldots,S_J\}}} \right). 
\end{eqnarray}
\end{corollary}


\subsection{Stationary distribution of the `natural' Markov chains}
\label{sec.natural}

We now consider the `natural' Markov chains $O$ of pair by pair and $Q^s$, $Q^c$ of server by server and customer by customer FCFS matching.  The state consists of the ordered list of unmatched customers and/or servers.

\begin{theorem}
The stationary distributions for $Q^s=(Q^s_N)_{N\in \bZ}$, for $Q^c=(Q^c_N)_{N\in \bZ}$ and for $O=(O_N)_{N\in \bZ}$ are given by
\begin{equation}
\label{eqn.stationaryQs}
\pi_{Q^s}(c^1,\ldots,c^L) = B \big(1- \beta_{\S(\{c^1,\ldots,c^L\})}\big) 
\prod_{\ell=1}^L \frac{\alpha_{c^\ell}}{\beta_{\S(\{c^1,\ldots,c^\ell\})}}, 
\end{equation}
\begin{equation}
\label{eqn.stationaryQc}
\pi_{Q^c}(s^1,\ldots,s^K) = B \big(1- \beta_{\C(\{s^1,\ldots,s^K\})}\big) 
\prod_{k=1}^K \frac{\beta_{s^k}}{\alpha_{\C(\{s^1,\ldots,s^k\})}}. 
\end{equation}
\begin{equation}
\label{eqn.stationaryO}
\pi_{O}(c^1,\ldots,c^L,s^1,\ldots,s^L) = B  
\prod_{\ell=1}^L  \frac{\alpha_{c^\ell}}{\beta_{\S(\{c^1,\ldots,c^\ell\})}}
\frac{\beta_{s^\ell}}{\alpha_{\C(\{s^1,\ldots,s^\ell\})}}.
\end{equation}
\end{theorem}
\begin{proof}
To prove (\ref{eqn.stationaryQs}), we need to sum $\pi_{Z^s}=\fz$ over all states $\fz$ of the form:
\[
c^1,s^{1,1},\ldots,s^{1,j_1},c^2,\ldots,c^\ell,s^{\ell,1},\ldots,s^{\ell,j_\ell},\ldots,c^L,s^{L,1},\ldots,s^{L,j_L},
\] 
where for $1\le \ell <L$ the range of $j_\ell$ is $0 \le j_\ell < \infty$ and for $j_L$ the range is $1 \le j_L < \infty$, 
and where each of $s^{\ell,m}$ ranges over possible values $s_j \not\in \S(\{c^1,\ldots,c^\ell\})$. 
We therefore have, by (\ref{eqn.stationrydist}), that 
\begin{eqnarray}
\label{eqn.summation}
&& \pi_{Q^s}(c^1,\ldots,c^L) = B \sum 
\alpha_{c^1}\beta_{s^{1,1}}\cdots\beta_{s^{1,j_1}}\alpha_{c^2}\cdots\alpha_{c^\ell}
\beta_{s^{\ell,1}}\cdots\beta_{s^{\ell,j_\ell}}\cdots\alpha_{c^L}\beta_{s^{L,1}}\cdots\beta_{s^{L,j_L}} 
\nonumber\\
&& \qquad = B \sum
\alpha_{c^1}\big(1- \beta_{\S(c^1)}\big)^{j_1} \alpha_{c^2} \cdots \alpha_{c^\ell}
\big(1 - \beta_{\S(\{c^1,\ldots,c^\ell\})}\big)^{j_\ell} \cdots \alpha_{c^L} \big(1 - \beta_{\S(\{c^1,\ldots,c^L\})}\big)^{j_L}   \nonumber \\
&& \qquad = B \frac{\alpha_{c^1}}{\beta_{\S(c^1)}} \cdots \frac{\alpha_{c^\ell}}{\beta_{\S(\{c^1\ldots,c^\ell\})}}
\cdots \frac{\alpha_{c^\ell}}{\beta_{\S(\{c^1\ldots,c^\ell\})}} 
(1 - \beta_{\S(\{c^1,\ldots,c^L\})}),  
\end{eqnarray}
which proves formula (\ref{eqn.stationaryQs}).  Formula (\ref{eqn.stationaryQc}) is proved in the same way.

To prove (\ref{eqn.stationaryO}), we need to sum $\pi_D=(\fz,\fy)$ over all states $(\fz,\fy)$ of the form:
\[
c^1,s^{1,1},\ldots,s^{1,j_1},c^2,\ldots,c^L,s^{L,1},\ldots,s^{L,j_L}
s^1,c^{1,1},\ldots,c^{1,k_1},s^2,\ldots,s^L,c^{L,1},\ldots,c^{L,k_L},
\] 
where for $1 \le \ell \le L$ the range of $j_\ell,k_\ell$ is $0 \le j_\ell,k_\ell < \infty$, and 
where each of $s^{\ell,m}$ ranges over possible values $s_j \not\in \S(\{c^1,\ldots,c^\ell\})$, and
each of $c^{\ell,m}$ ranges over possible values $c_i \not\in \C(\{s^1,\ldots,s^\ell\})$.
Performing the summation as in (\ref{eqn.summation}) we obtain (\ref{eqn.stationaryO}).
\end{proof}


\subsubsection{An example: the "NN" system}
\label{sec.example}
We consider the ``NN'' system, which is illustrated in Figure \ref{fig.NNsystem}.  This system was studied in \cite{caldentey-kaplan-weiss:09}, 
\begin{figure}[htb]
\begin{center}
\includegraphics[scale=0.30]{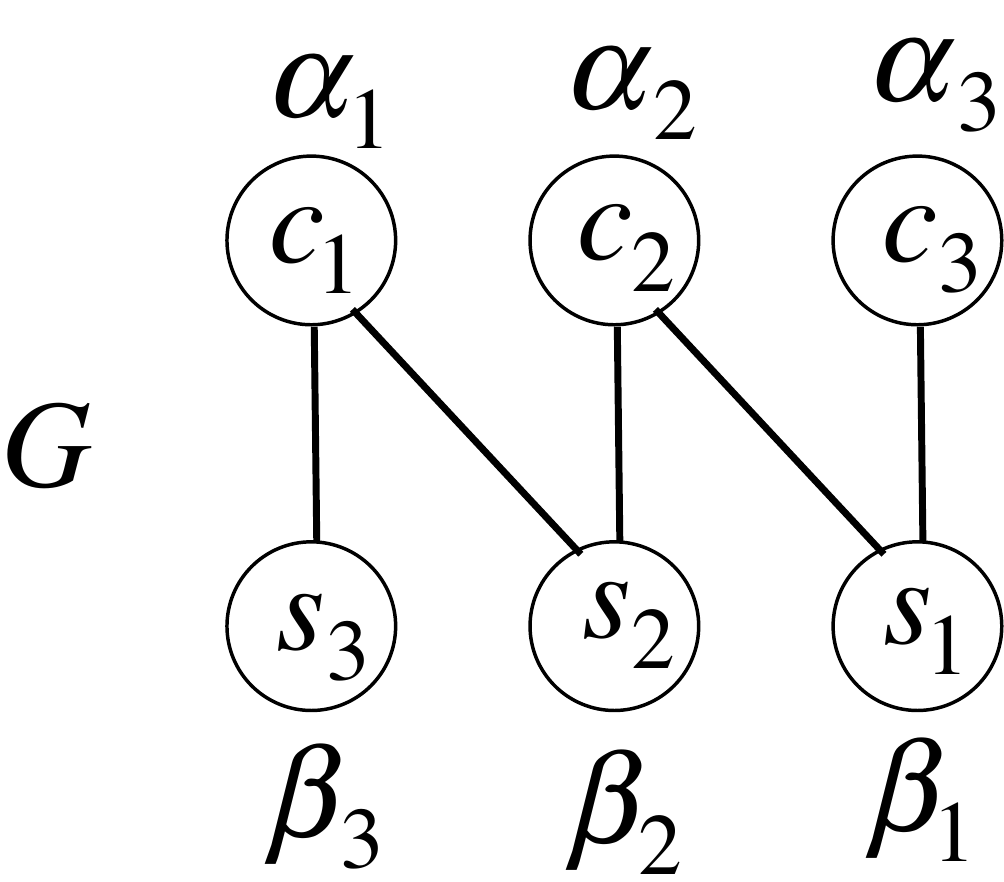}
\end{center}
\caption{Compatibility graph and probabilities of the  ``NN'' system}
\label{fig.NNsystem}
\end{figure}
where ergodicity under complete resource pooling was demonstrated but stationary probabilities could not at that time be obtained.

We can easily calculate stationary probabilities of the pair by pair,  server by server and customer by customer `natural' FCFS processes, using formulae (\ref{eqn.stationaryQs})--(\ref{eqn.stationaryO}), (\ref{eqn.B}).  The conditions for stability, i.e. for complete resource pooling, are:
\[
\beta_1 > \alpha_3,  \qquad  \alpha_1 >  \beta_3, \qquad  \alpha_1 + \beta_1 <1.
\]

Some examples are (we write $\alpha_i,\beta_j$ for $\alpha_{c_i},\beta_{s_j}$):  
\begin{eqnarray*}
&& P(Q^s_N=c_1,c_1,c_1,c_1) = B \beta_1 \left( \frac{\alpha_1}{\beta_2+\beta_3} \right)^4,
\\
&& P(Q^s_N=c_3,c_3,c_3,c_3,c_3) = B (1-\beta_1) \left( \frac{\alpha_3}{\beta_1} \right)^5,
\\
&& P(Q^c_N= s_3,s_3,s_3,s_2,s_3,s_2) = B \alpha_3 \left( \frac{\beta_3}{\alpha_1} \right)^3 
\left( \frac{\beta_2}{\alpha_1+\alpha_2} \right)^2 \frac{\beta_3}{\alpha_1+\alpha_2}, 
\\
&& P\big(O_N = (c_3,c_3,c_2,c_3,c_2,c_3),(s_3,s_3,s_3,s_3,s_3,s_3)\big) = 
B \left( \frac{\alpha_3}{\beta_1} \right)^2 \left( \frac{\alpha_2}{\beta_1+\beta_2} \right)^2  
\left( \frac{\alpha_3}{\beta_1+\beta_2} \right)^2 
\left( \frac{\beta_3}{\alpha_1} \right)^6. 
\end{eqnarray*}
The value of the normalizing constant is:
\[
B = \frac{(\alpha_1-\beta_3)(\beta_1-\alpha_3)(1 -\alpha_1 - \beta_1)}{\alpha_1 \alpha_2 \beta_1 \beta_2}.
\]


\section{Calculation of performance measures}
\label{sec.perfomance}
\subsection{Matching rates}
\label{sec.matchingrates}
Assume ergodicity (complete resource pooling) holds.  
The {\em matching rate} $r_{c_i,s_j}$ is the a.s. limit of the fraction of matches of customer of type $c_i$ with server of type $s_j$, in the complete FCFS  matching of $(s^n,c^m)_{0\le m,n \le N}$, as $N\to\infty$.  
An expression for $r_{c_i,s_j}$ was derived in \cite{adan-weiss:11}.  We include this expression here for completeness and also because of its close similarity to the derivation of the distribution of {\em link lengths} $L_{s_j},\,L_{c_i,s_j}$, in Section \ref{sec.linklengths}.  

Both $r_{c_i,s_j}$ and the distribution of $L_{c_i,s_j}$ are obtained by considering the state of the process $^o Z^s$ which is $\fs=(S_1,\fw_1,S_2,\fw_2,\ldots,S_J,\fw_J)$  (or equivalently, of the process $\Z$ with the addition of $\fw_J$, where $\fw_J$ is the sequence of $c^{\oM+1},\ldots,c^\oN$).  The final expressions include summation over all the permutations $S_1,\ldots,S_J \in \P_J$ of the  servers $s_1,\ldots,s_J$.

For convenience we use the following notations relative to each permutation $S_1,\ldots,S_J$:  
\[
\alpha_{(k)} = \alpha_{U\{S_1,\ldots,S_k\}}, \qquad \beta_{(k)} =
\beta_{\{S_1,\ldots,S_k\}}=\beta_{S_1}+\cdots+\beta_{S_k} .
\]
Note that if $U\{S_1,\ldots,S_k\}=\emptyset$ then $\alpha_{(k)}=0$.  Further,
\[
\phi_k = \frac{\alpha_{U\{S_1,\ldots,S_k\}\cap \{c_i\}}}{\alpha_{U\{S_1,\ldots,S_k\}}}, \qquad
\psi_k = \frac{\alpha_{U\{S_1,\ldots,S_k\}\cap (C(s_j)\backslash\{c_i\})}}{\alpha_{U\{S_1,\ldots,S_k\}}}, \qquad
\chi_k=1-\phi_k-\psi_k ,
\]
where $\phi_k,\psi_k,\chi_k$ express the conditional probability that $s^{N+1}=s_j$ and $c^m\in \fw_k$ form an $(s_j,c_i)$ match, or an $(s_j,c_k),\,c_k\ne c_i$ match, or no match at all, respectively.  By convention $0/0=0$.

The expression for the matching rate is:
\begin{eqnarray}
\label{eqn.matches}
r_{c_i,s_j} &=&  \beta_{s_j} \sum_{(S_1,\ldots,S_J)\in \P_J}  \pi_R( S_1,\ldots,S_J)  \nonumber \\
&& \left(
\sum_{k=1}^{J-1} \phi_k
\frac{\alpha_{(k)}}{\beta_{(k)}-\alpha_{(k)} \chi_{k}}
\prod_{l=1}^{k-1} \frac{\beta_{(l)}-\alpha_{(l)}}{\beta_{(l)}-\alpha_{(l)} \chi_{l}}
+ \frac{\phi_J}{\phi_J+\psi_J}
\prod_{l=1}^{J-1} \frac{\beta_{(l)}-\alpha_{(l)}}{\beta_{(l)}-\alpha_{(l)} \chi_{l}}
\right) .
\end{eqnarray}


\subsection{Link lengths}
\label{sec.linklengths}
Assume ergodicity (equivalently, complete resource pooling) holds, and consider a stationary FCFS matching over $\bZ$.  
For each server $s^n$, if $s^n$ is matched to $c^m$ we let $L(s^n,c^m)=m-n$ denote the link length.  We define the random variable $L_{s_j}$ to have the stationary distribution of link lengths of servers of type $s_j$.  We define the random variable $L_{s_j,c_i}$ to have the stationary distribution of link lengths of matches between servers of type $s_j$ and customers of type $c_i$.  In this section we derive the distributions of 
 $L_{s_j}$ and of $L_{s_j,c_i}$.  They are mixtures of convolutions of some positively signed and some negatively signed geometric random variables.

Consider the system following server by server FCFS matching and exchanging of all servers up to $N$.
The state is $\fs= \ts^\uN,\ldots,c^\oN = S_1,\fw_1,\ldots,S_J,\fw_J$.  Let $s^{N+1}=s_j$.  
$s^{N+1}$ will be matched to one of the elements of $\fs$, since $\fs$ contains customers of all types.  
Say it is matched to $c^m$.  We are interested in $L_{s_j} = m-(N+1)$.  
We are also interested in the special case that $c^m=c_i$.  The conditional random variable, conditional on   $c^m = c_i$ is then $L_{c_i,s_j}=m-(N+1)$.

We introduce the following additional notation.  For the state $\fs$, we let $\sharp c$ and $\sharp \ts$ count respectively the number of unmatched customers and the number of matched exchanged servers in $\fs$.   Also for $c^m$ we let $\sharp c_L$, $\sharp c_R$, $\sharp \ts_L$, $\sharp \ts_R$   count respectively the number of customers and of servers,  to the left (preceding) and to the right (succeeding) of $c^m$ in $\fs$.

\begin{proposition}  
If $s^{N+1}$ is matched to $c^m$, then $L(s^{N+1},c^m)=m-(N+1)$ is equal to the total number of unmatched customers preceding $c^m$ minus the total number of matched exchanged servers following $c^m$, that is:
\begin{equation}
\label{eqn.linkcount}
L(s^{N+1},c^m)=m-(N+1) =\sharp c_L - \sharp \ts_R.
\end{equation}
\end{proposition}  
\begin{proof}
 Figure \ref{fig.link} describes a link from $s^{N+1}$ to $c^m$.  
 \begin{figure}[htb]
\begin{center}
\includegraphics[scale=0.30]{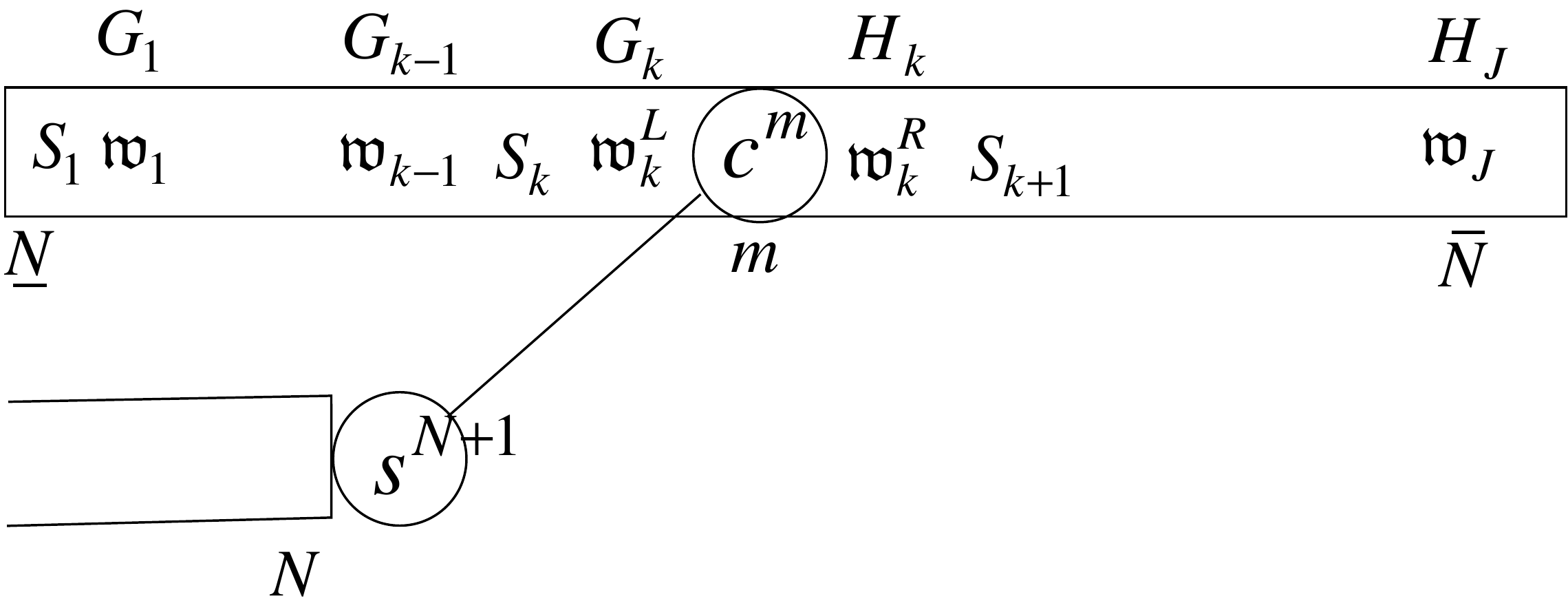}
\end{center}
\caption{The length of the link $L(s^{N+1},c^m)$, from  $s^{N+1}$ to $c^m$}
\label{fig.link}
\end{figure}
 In this figure we see the position of $N$ relative to $\uN$ and $\oN$, and the position $m$.  Conditional on $S_1,\ldots,S_J$ we have the words inbetween these last servers of each type, $\fw_1,\ldots,\fw_J$.  The match occurs in position $m$ in the middle of the word $\fw_\ell$.  $G_k$ denotes the random number of customers in the word $\fw_k$ skipped by $s^{N+1}$ in the search for a match, where $k=1,\ldots,\ell$.  We will see that  $G_k$ are independent and geometrically distributed.  $H_k$  is the random number of matched and exchanged servers, occurring in the word $\fw_k$ after the position $m$, for $k=\ell,\ldots,J$.  We will see that $H_k$ are also independent  and geometrically distributed, and $G_1,\ldots,G_\ell$ and $H_\ell,\ldots,H_J$ are independent.
In this figure,  $\sharp c_L = G_1 + \cdots +G_\ell$ and $\sharp s_R = H_\ell + \cdots +H_J$
 
The lemma follows from:
\begin{eqnarray*}
&& N - \uN +1 = \sharp \ts, \qquad  \sharp \ts = \sharp \ts_L + \sharp \ts_R, \qquad
 m- \uN = \sharp c_L + \sharp \ts_L, \\
&& L(s^{N+1},c^m) = m - (N+1) = m - \uN - \sharp \ts =  \sharp c_L - \sharp \ts_R.
\end{eqnarray*}
\end{proof}

We now need to count the number of matched customers preceding $c^m$, and matched and exchanged servers succeeding $c^m$.   We condition on $s^{N+1}=s_j$, and on $R_N=(S_1,\ldots,S_J)$.  Then $\fw_k$ consists of sequences of skipped customers of type $\U\{S_1,\ldots,S_k\}$, and exchanged servers of type $S_{k+1},\ldots,S_J$, and their number is geometrically distributed (with values $0,1,\ldots)$ as in (\ref{eqn.geometric-c}), (\ref{eqn.geometric-s}): 
\[
\sharp(c^r,\fw_k)\sim Geom_0\left(1 - \frac{\alpha_{(k)}}{\beta_{(k)}}\right), \qquad
\sharp(\ts^h,\fw_k)\sim Geom_0\left(1 - \frac{1 - \beta_{(k)}}{1-\alpha_{(k)}}\right)
\]

\begin{figure}[htb]
\begin{center}
\includegraphics[scale=0.40]{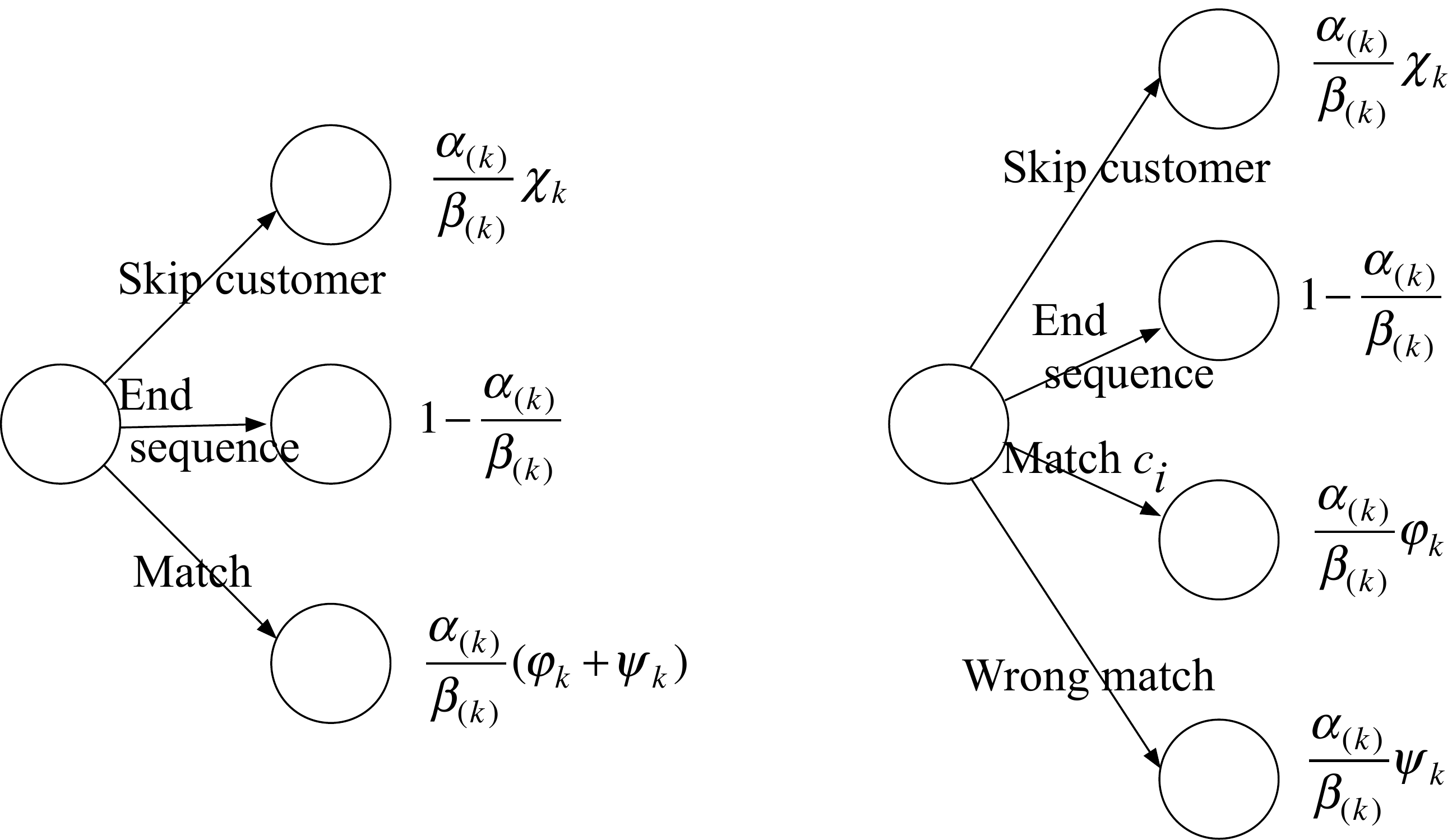}
\end{center}
\caption{Probabilities in searching for a match in $\fw_k$}
\label{fig.trees}
\end{figure}
We now count  $\sharp c_L$, the number of skipped customers prior to a match $c^m$.  Within the word $\fw_k$, following each skipped customer or at the beginning of the word, we can encounter either a next skipped customer, or the end of the word, or a match for $s^{N+1}=s_j$.  When a match occurs it can be with $c^m=c_i$, or it can be with $c^m \in \C(s_j)\backslash c_i$.   The probabilities for these outcomes are given in Figure \ref{fig.trees}.

Consider the event that $s^{N+1}$ finds a match of any type within $\fw_\ell$, after skipping 
$r_1,\ldots,r_{\ell-1},r_\ell$ customers in $\fw_1,\ldots,\fw_\ell$.  The probability of this event is:
\[
\left[ \prod_{k=1}^{\ell-1} \left( \frac{\alpha_{(k)}}{\beta_{(k)}} \chi_k \right)^{r_k} 
\left(1 - \frac{\alpha_{(k)}}{\beta_{(k)}} \right) \right] 
\left( \frac{\alpha_{(\ell)}}{\beta_{(\ell)}} \chi_\ell \right)^{r_\ell} 
\frac{\alpha_{(\ell)}}{\beta_{(\ell)}} (\phi_\ell + \psi_\ell ),
\]
Dividing and multiplying by $1 - \frac{\alpha_{(k)}}{\beta_{(k)}} \chi_k,\; k=1,\ldots,\ell$ and rearranging we obtain:
\[
\left( \frac{ \alpha_{(\ell)} (\phi_\ell + \psi_\ell ) }{\beta_{(\ell)} -\alpha_{(\ell)} \chi_\ell } 
\prod_{k=1}^{\ell-1} \frac{ \beta_{(k)} - \alpha_{(k)} }{ \beta_{(k)} - \alpha_{(k)} \chi_k } \right)
 \prod_{k=1}^\ell \left(\frac{\alpha_{(k)}}{\beta_{(k)}} \chi_k \right)^{r_k} 
\left(1 - \frac{\alpha_{(k)}}{\beta_{(k)}} \chi_k \right),
\]
here the $r_k$ are seen to have a geometric distribution, and the probability multiplying the geometric distribution terms is the probability that the match will occur in the word $\fw_\ell$.

Consider next the same event conditional on the match being of type $c_i$.  We need to divide the probabilities in each step by $1- \frac{\alpha_{(k)}}{\beta_{(k)}} \psi_k$.  The resulting probability is:
\[
\left[ \prod_{k=1}^{\ell-1} \left( \frac{\alpha_{(k)} \chi_k}{\beta_{(k)} - \alpha_{(k)} \psi_k}  \right)^{r_k} 
\left(\frac{\beta_{(k)} - \alpha_{(k)}}{\beta_{(k)} - \alpha_{(k)} \psi_k} \right) \right] 
\left( \frac{\alpha_{(\ell)} \chi_\ell}{\beta_{(\ell)} - \alpha_{(\ell)} \psi_\ell}  \right)^{r_\ell}
\frac{\alpha_{(\ell)} \phi_\ell}{\beta_{(\ell)} - \alpha_{(\ell)} \psi_\ell},
\]
Again, dividing and multiplying by $1 - \frac{\alpha_{(k)} \chi_k }{\beta_{(k)}- \alpha_{(k)} \psi_k }\; k=1,\ldots,\ell$ and rearranging we obtain:
\[
\left( \frac{ \alpha_{(\ell)} \phi_\ell }{ \beta_{(\ell)} - \alpha_{(\ell)}  (\psi_\ell + \chi_\ell) } 
\prod_{k=1}^{\ell-1} \frac{ \beta_{(k)} -  \alpha_{(k)} }{ \beta_{(k)} - \alpha_{(k)} (\psi_k + \chi_k) } 
\right)  \prod_{k=1}^\ell \left(\frac{\alpha_{(k)} \chi_k }{\beta_{(k)}- \alpha_{(k)} \psi_k }  \right)^{r_k} 
\left(1 - \frac{\alpha_{(k)} \chi_k }{\beta_{(k)}- \alpha_{(k)} \psi_k }  \right),
\]
where again we have $r_k$ with geometric distributions, multiplied by the probability that the match with $c_i$ happens in word $\fw_\ell$.

We see that conditional on $S_1,\ldots,S_J$  the value of $\sharp c_L$ in the match of $s^{N+1}=s_j$ with $c^m$,  is a mixture of convolutions of $\ell$ geometric random variables, where the mixture is over the value of $\ell$.  We use $\bigstar _{k=1}^\ell$ to denote convolution of 
 $\ell$ random variables indexed by $k$:
\begin{equation}
\label{eqn.convolution1}
\sharp c_L \sim G_1\star\cdots\star G_\ell = \bigstar _{k=1}^\ell Geom_0 \left (1 - \frac{\alpha_{(k)}}{\beta_{(k)}} \chi_k \right)
\mbox{ w.p. }
\frac{ \alpha_{(\ell)} (\phi_\ell + \psi_\ell ) }{\beta_{(\ell)} -\alpha_{(\ell)} \chi_\ell } 
\prod_{k=1}^{\ell-1} \frac{ \beta_{(k)} - \alpha_{(k)} }{ \beta_{(k)} - \alpha_{(k)} \chi_k } 
\end{equation}
and in the match  of $s^{N+1}=s_j$ with $c^m = c_i $:
\begin{equation}
\label{eqn.convolution2}
\sharp c_L \sim G_1\star\cdots\star G_\ell = \bigstar _{k=1}^\ell Geom_0 \left 
(1 - \frac{\alpha_{(k)}\chi_k}{\beta_{(k)} -\alpha_{(k)}\psi_k } \right)
\mbox{ w.p. }
\frac{ \alpha_{(\ell)} \phi_\ell }{ \beta_{(\ell)} - \alpha_{(\ell)}  (\psi_\ell + \chi_\ell) } 
\prod_{k=1}^{\ell-1} \frac{ \beta_{(k)} \alpha_{(k)} }{ \beta_{(k)} - \alpha_{(k)} (\psi_k + \chi_k) } 
\end{equation}

The number of matched and exchanged servers succeeding the match, $\sharp \ts_R$, is independent of the number of skipped customers prior to the match.  Also, if the match occurs in $\fw_\ell$, the number of matched and exchanged servers after $c^m$ is geometric by the memoryless property of geometric random variables.
Hence, conditional on $S_1,\ldots,S_J$ in the match of $s^{N+1}=s_j$ with $c^m$,  $\sharp s_R$ is a mixture of convolutions of geometric random variables, plus the succeeding servers $S_{\ell+1},\ldots,S_J$
\[
\sharp \ts_R \sim H_\ell\star\cdots\star H_J + (J-\ell) = (J-\ell)+ \bigstar _{k=\ell}^J Geom_0 \left (1 - \frac{1-\beta_{(k)}}{1-\alpha_{(k)}}  \right)
\begin{array}{c} \mbox{with appropriate probabilities}\\ 
\mbox{as in (\ref{eqn.convolution1}), (\ref{eqn.convolution2})}\end{array}
\]
{\bf Note}:  In all these equations, empty sums, empty products, and 0 probabilities occur when necessary.

We summarize these results in terms of  generating functions.
\begin{theorem}
The generating functions of the distributions of $L_{s_j},\,L_{c_i,s_j}$ are:
\begin{eqnarray}
\label{eqn.generating1}
&& E(\sZ^{L_{s_j}}) = \sum_{(S_1,\ldots,S_J)\in \P_J}  \pi_R( S_1,\ldots,S_J) \sum_{\ell = 1}^J 
 \frac{ \alpha_{(\ell)} (\phi_\ell + \psi_\ell ) }{\beta_{(\ell)} -\alpha_{(\ell)} \chi_\ell } 
 \prod_{k=1}^{\ell-1} \frac{ \beta_{(k)} - \alpha_{(k)} }{ \beta_{(k)} - \alpha_{(k)} \chi_k } 
 \nonumber \\
&&\qquad 
\times \prod_{k=1}^\ell \Big( \frac{ \beta_{(k)} - \alpha_{(k)} \chi_k } 
{ \beta_{(k)} - \alpha_{(k)} \chi_k \sZ } \Big)
\times \prod_{k=\ell}^J  \Big( \frac{ \beta_{(k)} - \alpha_{(k)} } 
{ 1-\alpha_{(k)} - (1-\beta_{(k)})  {\sZ}^{-1} } \Big) \times \frac{1}{\sZ^{J-\ell}} 
\end{eqnarray}
\begin{eqnarray}
\label{eqn.generating2}
&& E(\sZ^{L_{s_j,c_i}}) = \sum_{(S_1,\ldots,S_J)\in \P_J}  \pi_R( S_1,\ldots,S_J) \sum_{\ell = 1}^J 
\frac{ \alpha_{(\ell)} \phi_\ell }{ \beta_{(\ell)} - \alpha_{(\ell)}  (\psi_\ell + \chi_\ell) }
\prod_{k=1}^{\ell-1} \frac{ \beta_{(k)} -  \alpha_{(k)} }{ \beta_{(k)} - \alpha_{(k)} (\psi_k + \chi_k) } 
\nonumber \\
&& \qquad 
\times \prod_{k=1}^\ell \Big( \frac{ \beta_{(k)} - \alpha_{(k)}(\psi_k + \chi_k) } 
{ \beta_{(k)} -\alpha_{(k)} \psi_k  - \alpha_{(k)} \chi_k \sZ } \Big)
\times \prod_{k=\ell}^J  \Big( \frac{ \beta_{(k)} - \alpha_{(k)} } 
{ 1-\alpha_{(k)} - (1-\beta_{(k)})  {\sZ}^{-1} } \Big) \times \frac{1}{\sZ^{J-\ell}} 
\end{eqnarray}
\end{theorem}

\appendix


\section{Appendix: Completion of the proof of Theorem \ref{thm.stationary}}
\label{sec.appendix}
\begin{proof}[Proof of part (ii) of Theorem \ref{thm.stationary}]
The proof is similar to the proof of part (i) of the theorem.  
We need to verify that the proposed $\pi_D,\,\pi_E$, similar to (\ref{eqn.checking1}),  satisfy:
\begin{eqnarray*}
\label{eqn.checking2}
&& \pi_{D}(\fz,\fy) P\big(D_{N+1}=(\fz',\fy')  \,|\,D_N=(\fz,\fy)\big) = 
\pi_{D}(\fz',\fy') P\big(D_n=(\fz,\fy)  \,|\,D_{N+1}=(\fz',\fy')\big) \\
&& \qquad = \pi_{E}(\overleftarrow{\fy'},\overleftarrow{\fz'}) 
P\big(E_{M+1} =(\overleftarrow{\fy},\overleftarrow{\fz}) \,|\, E_M =(\overleftarrow{\fy'},\overleftarrow{\fz'}\big)).
\end{eqnarray*}

There are 10 types of transitions from $D_N$ to $D_{N+1}$ and from $E_M$ to $E_{M+1}$, and some special cases when the matching is perfect.  We verify (\ref{eqn.checking2}) for a few of them, the verification for the remaining types is similar.

$(i)$ In the transition from $D_N$ to $D_{N+1}$, customer $c^{N+1}=c_i$ is matched and exchanged with  some $y^k=s_h$  where $k>1$, while server  $s^{N+1}=s_j$ remains unmatched.  This happens if  $c_i \in \overline{\C}\{y^1,\ldots,y^{k-1}\}\cap \C(y^k)$, and $s_j \not\in \S\{z^1,\ldots,z^L\}$.   The transition probability is  $\alpha_{c_i} \beta_{s_j}$.  In this transition the length of $y$ and $z$  is increased by 1.  
The transition is 
\[
D_N = \left\{ \begin{array}{r} z^1,\ldots,z^L\\ y^1,\ldots,y^k=s_h,\ldots,y^K \end{array}\right\} \to 
 D_{N+1}=  \left\{ \begin{array}{r} z^1,\ldots,z^L,s_h \\ y^1,\ldots,y^k=c_i,\ldots,y^K,s_j \end{array}\right\}. 
\]

In the reversed transition from $\overleftarrow{E}_{N+1}$ to $\overleftarrow{E}_N$, we see: 
\[
\overleftarrow{E}_{N+1} =  \left\{ \begin{array}{r} z^1,\ldots,z^L,s_h \\ y^1,\ldots,y^k=c_i,\ldots,y^K,s_j \end{array}\right\} \to  
 \overleftarrow{E}_N \left\{ \begin{array}{r} z^1,\ldots,z^L\\ y^1,\ldots,y^k=s_h,\ldots,y^K \end{array}\right\}, 
\]
which is the same as the transition
\[
E_M =  \left\{ \begin{array}{l} s_j,y^K,\ldots,y^k=c_i,\ldots,y^1 \\ z_h,z^L,\ldots,z^1  \end{array}\right\} \to  
E_{M+1} \left\{ \begin{array}{l} y^K,\ldots,y^k=s_h,\ldots,y^1 \\ z^L,\ldots,z^1 \end{array}\right\}. 
\]
The length of $y$ and $z$ is reduced by 1, and this transition is deterministic,  it happens with probability 1.

Verification of  (\ref{eqn.checking2}) follows from:
\[
\pi_D(D_{N+1}) = \pi_D(D_N) \alpha_{c_i} \beta_{s_j}, \quad   P(D_{N+1}|D_N) = \alpha_{c_i} \beta_{s_j},   
\quad P(\overleftarrow{E}_N | \overleftarrow{E}_{N+1}) =1.
\]

$(ii)$   In the transition from $D_N$ to $D_{N+1}$, customer $c^{N+1}=c_i$ is matched and exchanged with   $y^1=s_h$, and $s^{N+1}=s_j$ is matched and exchanged with $z^1=c_t$.   The resulting transition is either:
\[
D_N = \left\{ \begin{array}{r} z^1=c_t,z^2,\ldots,z^L\\ y^1=s_h,y^2,\ldots,y^K \end{array}\right\} \to 
 D_{N+1}=  \left\{ \begin{array}{r} z^l,\ldots,z^L,s_h \\ y^k,\ldots,y^K,c_t \end{array}\right\}, 
\]
or
\[
D_N = \left\{ \begin{array}{r} z^1=c_t,z^2,\ldots,z^L\\ y^1=s_h,y^2,\ldots,y^K \end{array}\right\} \to 
 D_{N+1}=  \emptyset,
\]
and the corresponding reversed transition is in the former case
\[
E_M =\left\{ \begin{array}{l}  c_t,y^K,\ldots,y^k \\ s_h,z^L,\ldots,z^l \end{array}\right\}  \to
E_{M+1} =  \left\{ \begin{array}{l}   y^K,\ldots,y^2,y^1=s_h \\ z^L,\ldots,z^2,z^1=c_t  \end{array}\right\},  
\]
and in the latter case
\[
E_M =\emptyset   \to
E_{M+1} =  \left\{ \begin{array}{l}   y^K,\ldots,y^2,y^1=s_h \\ z^L,\ldots,z^2,z^1=c_t  \end{array}\right\}.  
\]
Verification of  (\ref{eqn.checking2}) follows in the former case from:
\begin{eqnarray*}
&\pi_D(D_{N+1}) = 
\pi_D(D_N) \Big{/} \beta_{z^2}\cdots \beta_{z^{l-1}}  \alpha_{y^2}\cdots \alpha_{y^{k-1}}, 
\quad    P(D_{N+1}|D_N) = \alpha_{\C(s_h)} \beta_{\S(c_t)},   &\\
& P(\overleftarrow{E}_N | \overleftarrow{E}_{N+1}) =   
\beta_{z^2}\cdots \beta_{z^{l-1}}  \beta_{\S(c_t)}  \alpha_{y^2}\cdots \alpha_{y^{k-1}} \alpha_{\C(s_h)},  &
\end{eqnarray*}
and in the latter case from:
\begin{eqnarray*}
&\pi_D(D_{N+1}) = 
\pi_D(D_N) \Big{/} \alpha_{c_t} \beta_{z^2}\cdots \beta_{z^L}  \beta_{s_h} \alpha_{y^2}\cdots \alpha_{y^K}, 
\quad    P(D_{N+1}|D_N) = \alpha_{\C(s_h)} \beta_{\S(c_t)},   &\\
& P(\overleftarrow{E}_N | \overleftarrow{E}_{N+1}) =   
\alpha_{c_t} \beta_{z^2}\cdots \beta_{z^L}  \beta_{\S(c_t)} 
\beta_{s_h} \alpha_{y^2}\cdots \alpha_{y^K} \alpha_{\C(s_h)}.  &
\end{eqnarray*}

$(iii)$  Transition from empty state to empty state, $D_N=D_{N+1}=E_M=E_{M+1}=\emptyset$:
\[
\pi(D_N)=\pi(D_{N+1}), \quad 
P(D_{N+1}|D_N)=P(E_{M+1}|E_M)=\sum_{(c_i,s_j)\in \E} \alpha_{c_i} \beta_{s_j}.
\]

$(iv)$  Transition from  empty state to state 
$D_{N+1}=\left( \begin{array}{l} c_i \\ s_j \end{array} \right)$.  The reversed transition is from
 $\overleftarrow{E}_{M+1}=\left( \begin{array}{l} c_i \\ s_j \end{array} \right)$ 
 to  $\overleftarrow{E}_M = \emptyset$, which is the same transition as 
 $E_M=\left( \begin{array}{l} s_j \\ c_i \end{array} \right)$ to $E_{M+1}=\emptyset$.  
 Recall that if $E_M=\left( \begin{array}{l} s_j \\ c_i \end{array} \right)$ then all customers and servers preceding position $M$ have been matched and exchanged, and in particular, $s_j = \ts^{M+1}$ and $c_i=\tc^{M+1}$ are matched and exchanged with the original $c^{M+1},s^{M+1}$ and so the transition from $E_M=\left( \begin{array}{l} s_j \\ c_i \end{array} \right)$ to $E_{M+1}=\emptyset$ is deterministic, and has probability 1.  To verify   (\ref{eqn.checking2}):
\[
\pi_D(D_N)=\pi_D(D_{N+1})\Big{/} \alpha_{c_i} \beta_{s_j}, \quad 
P(D_{N+1}|D_N)=\alpha_{c_i} \beta_{s_j}, \quad 
P(E_{M+1}|E_M)=1.
\]
\end{proof}

\end{document}